\documentclass[pdftex]{amsart}

\address{Department of Mathematics, Tokyo Institute of Technology, 2-12-1 Ookayama, Meguro-ku, Tokyo, 152-8551, Japan}
\email{isoshima.t.aa@m.titech.ac.jp}

\usepackage{amsmath,amssymb}
\usepackage{amsfonts}
\usepackage{graphicx}

\usepackage{hyperref}
\usepackage[hang,small,bf]{caption}
\usepackage[subrefformat=parens]{subcaption}
\usepackage[final]{pdfpages}
\usepackage{amsthm}
\usepackage{color}

\theoremstyle{plain}
\newtheorem{thm}{Theorem}[section]
\newtheorem{prop}[thm]{Proposition}
\newtheorem{lem}[thm]{Lemma}

\newtheorem*{thm*}{Theorem}
\newtheorem*{cor*}{Corollary}

\theoremstyle{definition}
\newtheorem{dfn}[thm]{Definition}
\newtheorem{rem}[thm]{Remark}
\newtheorem{exm}[thm]{Example}
\newtheorem{que}[thm]{Question}

\newtheorem{con}[thm]{Conjecture}
\newtheorem*{que*}{Question}
\newtheorem*{con*}{Conjecture}

\begin{document}

\title{Trisections of the doubles of some Mazur type 4-manifolds}
\author{Tsukasa Isoshima}
\date{}

\begin{abstract}
We show that certain two kinds of trisection diagrams of the doubles of the Mazur type 4-manifolds introduced by Akbulut and Kirby are standard. One is constructed by doubling a certain relative trisection diagram of the Mazur type. The other is constructed by using an algorithm taking Kirby diagrams to trisection diagrams.
\end{abstract}

\maketitle

\section{Introduction}\label{sec:intro}
A trisection, introduced by Gay and Kirby \cite{MR3590351}, is roughly speaking a decomposition of a 4-manifold into three 4-dimensional 1-handlebodies. Any 4-manifold admits a trisection, and two trisections of the same 4-manifold are isotopic after some stabilizations. In the present, the following is conjectured as a 4-dimensional analogue of Waldhausen's theorem in the theory of Heegaard splittings which states that each Heegaard splitting of $S^3$ is isotopic to the stabilization of the genus 0 Heegaard splitting.

\begin{con}[{\cite[Conjecture 3.11]{MR3544545}}]
Each trisection of $S^4$ is isotopic to the stabilization of the genus 0 trisection.
\end{con}

Trisection diagrams are new diagrams describing 4-manifolds. A trisection diagram consists of an orientable closed surface and certain three kinds of cut systems on the surface called $\alpha$ (red), $\beta$ (blue) and $\gamma$ (green) curves (see Figure \ref{fig:modeldiagramofrelativetrisection} for example). Any 4-manifold can be described by a trisection diagram. Trisection diagrams can be defined independently of trisections although they have the word, trisection, in the name, but there is a one to one correspondence between trisections and trisection diagrams under an appropriate equivalence relation. For example, two trisections are diffeomorphic if and only if two corresponding trisection diagrams are related by only surface diffeomorphisms and handle slides among the same family curves. 

In this paper, as a special case of the above conjecture, we consider the following question. Note that we call a trisection diagram of a 4-manifold diffeomorphic to $S^4$ \textit{standard} if it is related to the stabilization of the genus 0 trisection diagram without stabilizations, or equivalently, it is related to the genus 0 trisection diagram by only surface diffeomorphisms, handle slides and destabilizations.

\begin{que}
Is a trisection diagram of a 4-manifold $X$ diffeomorphic to $S^4$ standard?
\end{que}

This question is based on the following result by Gay and Kirby \cite[Corollary 12]{MR3590351}: Two closed 4-manifolds are diffeomorphic if and only if corresponding two trisection diagrams are related by surface diffeomorphisms, handle slides among the same family curves and balanced (de)stabilizations. Note that it follows from the existence and uniqueness of trisections.

In \cite{https://doi.org/10.48550/arxiv.2205.04817}, the author considered the trivial Price twisted 4-manifold as $X$. The author and Ogawa \cite{isoshima2023trisections,isoshima2023infinitely} also considered the non-trivial Gluck twisted 4-manifold as $X$. In this paper, we consider as $X$ the double of the Mazur type $W^{\pm}(\ell,k)$ introduced by Akbulut and Kirby \cite{MR0544597} for any integers $\ell$ and $k$. We call a contractible 4-manifold with boundary of Mazur type if it admits a handle decomposition consisting of only single 0-, 1-, and 2-handle. It is known \cite{MR0125574} that the double of each Mazur type is diffeomorphic to $S^4$. Thus, we can take such a 4-manifold as $X$.

In this paper, we construct two kinds of trisection diagrams of the double $DW^{-}(0,n+2)$ of the Mazur type $W^{-}(0,n+2)$ in a different way. Firstly, we construct a trisection diagram of $DW^{-}(0,n+2)$ from a relative trisection diagram constructed by Takahashi \cite{takahashi2023exotic} (see Figure \ref{fig:Takahashi}) by using a method of \cite[Corollary 2.13]{MR3999550} which obtains trisection diagrams of doubles by doubling relative trisection diagrams. Figure \ref{fig:DDn} describes the trisection diagram. Our first result is as follows:

\begin{thm*}[Theorem \ref{thm1}]
The trisection diagram depicted in Figure \ref{fig:DDn} is standard for any integer $n$.
\end{thm*}

We also construct another trisection diagram of $DW^{-}(0,n+2)$ using an algorithm taking Kirby diagrams to trisection diagrams developed by Kepplinger \cite{MR4460230}. Figure \ref{fig:DDn2} describes the trisection diagram if we use Figure \ref{fig:DW-(l,k)} as an input Kirby diagram in the algorithm. Our second result is as follows:

\begin{thm*}[Theorem \ref{thm2}]
The trisection diagram depicted in Figure \ref{fig:DDn2} is standard for any integer $n$.
\end{thm*}

We proof the two theorems by performing so many Dehn twists and handle slides. Note that the trisection diagram constructed by the second method cannot be constructed by the first method. See Remark \ref{rem:main} for the proof.

This paper is organized as follows: In Section \ref{sec:preliminaries}, we review some notions used to mention the main theorems. In Section \ref{sec:main}, we depict the trisection diagrams of  $DW^{-}(0,n+2)$ explicitly and show that they are standard.

\section*{Acknowledgement}
The author would like to express sincere gratitude to his supervisor, Hisaaki Endo, for his helpful comments. The author also thanks Masaki Ogawa and Natsuya Takahashi for many discussions. The author was partially supported by JSPS KAKENHI Grant Number JP23KJ0888.

\section{Preliminaries}\label{sec:preliminaries}
In this paper, we suppose that any 4-manifold is compact, connected, oriented and smooth. The 4-manifold with opposite orientation of a 4-manifold $X$ is denoted by $\overline{X}$, and it is denoted by $X \cong Y$ that 4-manifolds $X$ and $Y$ are diffeomorphic.

\subsection{Trisections and relative trisections}
In this subsection, we recall trisections of closed 4-manifolds and relative trisections of 4-manifolds with boundary.

\begin{dfn}
Let $X$ be a closed 4-manifold. A $(g;k_1,k_2,k_3)$-\textit{trisection} of $X$ is a 3-tuple $(X_1,X_2,X_3)$ satisfying the following conditions:
\begin{itemize}
\item $X=X_1 \cup X_2 \cup X_3$,
\item For each $i=1,2,3$, $X_i \cong \natural_{k_i} S^1 \times D^3$,
\item For each $i=1,2,3$, $X_i \cap X_j \cong \natural_{g} S^1 \times D^2$,
\item $X_1 \cap X_2 \cap X_3 \cong \#_{g} S^1 \times S^1 = \Sigma_g$.
\end{itemize}
\end{dfn}

Let $H_{\alpha} = X_3 \cap X_1$, $H_{\beta} = X_1 \cap X_2$ and $H_{\gamma} = X_2 \cap X_3$. The union $H_{\alpha} \cup H_{\beta} \cup H_{\gamma}$ is called the \textit{spine} of the trisection. A trisection is uniquely determined by its spine. Note that if $k_1=k_2=k_3$, the trisection is said to be \textit{balanced}. And the 4-tuple of non-negative integers $(g;k_1,k_2,k_3)$ is called the \textit{type} of the trisection. 

\begin{exm}
The natural decomposition of $S^4$ into three 4-balls is the $(0,0)$-trisection (genus 0 trisection).
\end{exm}


\begin{dfn}
A 4-tuple $(\Sigma_g,\alpha,\beta,\gamma)$ is called a $(g;k_1,k_2,k_3)$-\textit{trisection diagram} if the following holds:
\begin{itemize}
\item $(\Sigma_g,\alpha,\beta)$ is a Heegaard diagram of $\#_{k_1}S^1 \times S^2$,
\item $(\Sigma_g,\beta,\gamma)$ is a Heegaard diagram of $\#_{k_2}S^1 \times S^2$,
\item $(\Sigma_g,\gamma,\alpha)$ is a Heegaard diagram of $\#_{k_3}S^1 \times S^2$.
\end{itemize}
\end{dfn}

\begin{exm}
For $\Sigma=S^2$ and $\alpha=\beta=\gamma=\emptyset$, $(\Sigma, \alpha,\beta,\gamma)$ is the genus 0 trisection diagram.
\end{exm}

\begin{rem}
For a genus $g$ trisection whose spine is $H_{\alpha} \cup H_{\beta} \cup H_{\gamma}$, let $\alpha$ (resp. $\beta$, $\gamma$) be the boundary of meridian disk systems of $H_\alpha$ (resp. $H_\beta$, $H_\gamma)$. Then, $(\Sigma_g,\alpha,\beta,\gamma)$ is the trisection diagram with respect to the trisection. Conversely, for a trisection diagram $(\Sigma_g,\alpha,\beta,\gamma)$, by attaching 2-handles to $\Sigma_g \times D^2$ along $\alpha \times \{e^{\frac{2{\pi}i}{3}}\}$, $\beta \times \{e^{\frac{4{\pi}i}{3}}\}$ and $\gamma \times \{e^{2{\pi}i}\}$ with surface framing, we can construct the trisected 4-manifold corresponding to the trisection diagram.
\end{rem}

\begin{dfn}
A \textit{stabilization} of a trisection diagram is the connected-sum of the trisection diagram and one of the genus 1 trisection diagram of $S^4$ depicted in Figure \ref{fig:stabilizationfordiagram}, or the trisection diagram itself obtained by the connected-sum. The reverse operation is called a \textit{destabilization}.
\end{dfn}

\begin{figure}[h]
\begin{center}
\includegraphics[width=8cm, height=3cm, keepaspectratio, scale=1]{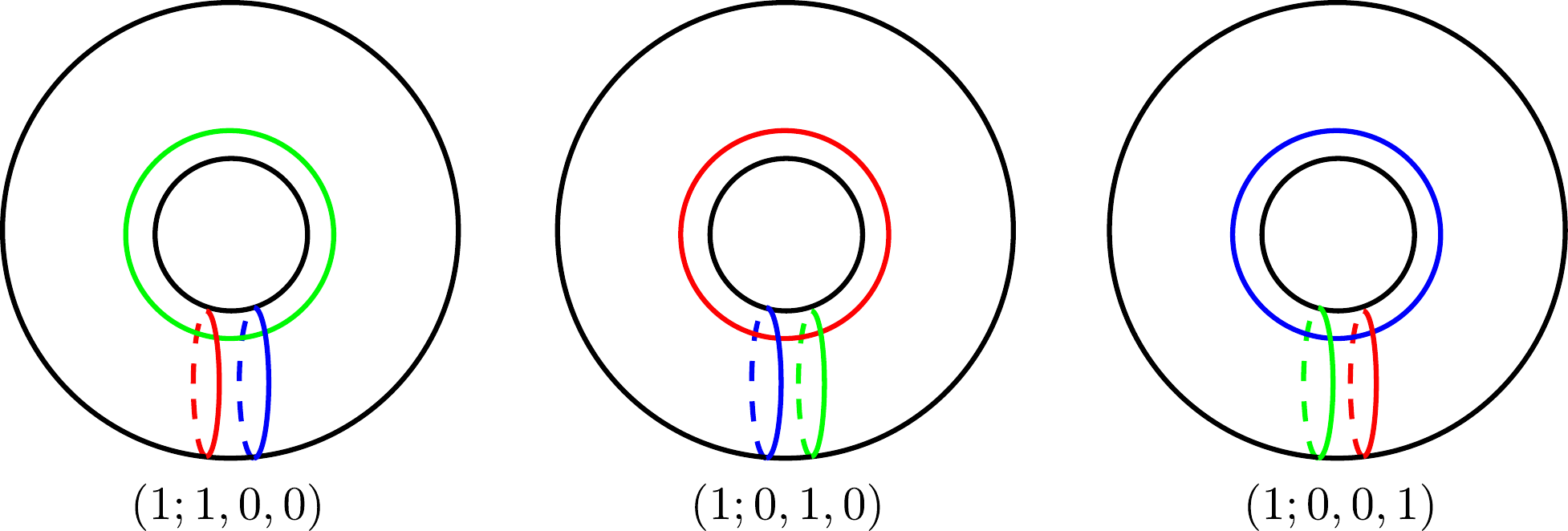}
\end{center}
\setlength{\captionmargin}{50pt}
\caption{The genus 1 trisection diagrams of $S^4$.}
\label{fig:stabilizationfordiagram}
\end{figure}

Note that this (de)stabilization is the unbalanced one. It is obvious that if one stabilizes a $(g;k_1,k_2,k_3)$-trisection diagram by the $(1;1,0,0)$-trisection diagram (resp. $(1;0,1,0)$- and $(1;0,0,1)$-), the type of the resulting trisection diagram is $(g+1,k_1+1,k_2,k_3)$ (resp. $(g+1,k_1,k_2+1,k_3)$ and $(g+1,k_1,k_2,k_3+1)$). Similarly, the type decreases in the case of destabilizations.

Any 4-manifold admits a trisection \cite[Theorem 4]{MR3590351}, and any two trisections of the same 4-manifold are isotopic after some balanced stabilizations \cite[Theorem 11]{MR3590351} (see below for the definition). As a corollary, it is shown that any two closed 4-manifolds are diffeomorphic if and only if corresponding two trisection diagrams are related by surface diffeomorphisms, handle slides among the same family curves (i.e. sliding $\alpha$, $\beta$ and $\gamma$ curves over $\alpha$, $\beta$ and $\gamma$ curves, respectively) and balanced stabilizations \cite[Corollary 12]{MR3590351}. 

Let $X$ and $Y$ be diffeomorphic closed 4-manifolds. Two trisections $(X_1, X_2, X_3)$ of $X$ and $(Y_1,Y_2,Y_3)$ of $Y$ are \textit{diffeomorphic} if there exists a diffeomorphism $h \colon X \to Y$ such that $h(X_i)=Y_i$ for each $i=1,2,3$. Two trisections $(X_1, X_2, X_3)$ and $(Y_1,Y_2,Y_3)$ of a 4-manifold $Z$ are \textit{isotopic} if there exists an isotopy $\{h_t\} \colon Z \to Z$ such that $h_0=id_{Z}$ and $h_1(X_i)=Y_i$ for each $i=1,2,3$. Two trisections are diffeomorphic if and only if corresponding two trisection diagrams are related by only surface diffeomorphisms and handle slides (namely, without stabilizations). It is conjectured that any trisection of $S^4$ is isotopic to the stabilization of the genus 0 trisection \cite[Conjecture 3.11]{MR3544545} (introduced in Section \ref{sec:intro}). This conjecture is called 4-dimensional Waldhausen's conjecture since it can be regarded as a 4-dimensional analogy of Waldhausen's theorem in the theory of Heegaard splittings, which states that any Heegaard splitting of $S^3$ is isotopic to the stabilization of the genus 0 Heegaard splitting. A counterexample of this conjecture may be given via trisection diagrams of a 4-manifold diffeomorphic to $S^4$ since any two isotopic trisections are diffeomorphic.

Trisections of 4-manifolds with boundary can be defined, which are called \textit{relative trisections}. See \cite{castro2016relative} for the definition. \textit{Relative trisection diagrams} can be defined as follows (for example see \cite{MR3770114}). Note that as well as the closed case, relative trisection diagrams can be defined independently of relative trisections.
\begin{dfn}
A $(g,k;p,b)$-\textit{relative trisection diagram} is a 4-tuple $(\Sigma, \alpha, \beta, \gamma)$ such that each of $(\Sigma, \alpha, \beta)$, $(\Sigma, \beta, \gamma)$ and $(\Sigma, \gamma, \alpha)$ is related to the diagram depicted in Figure \ref{fig:modeldiagramofrelativetrisection} by surface diffeomorphisms and handle slides.
\end{dfn}

\begin{figure}[h]
\begin{center}
\includegraphics[width=10cm, height=8cm, keepaspectratio, scale=1]{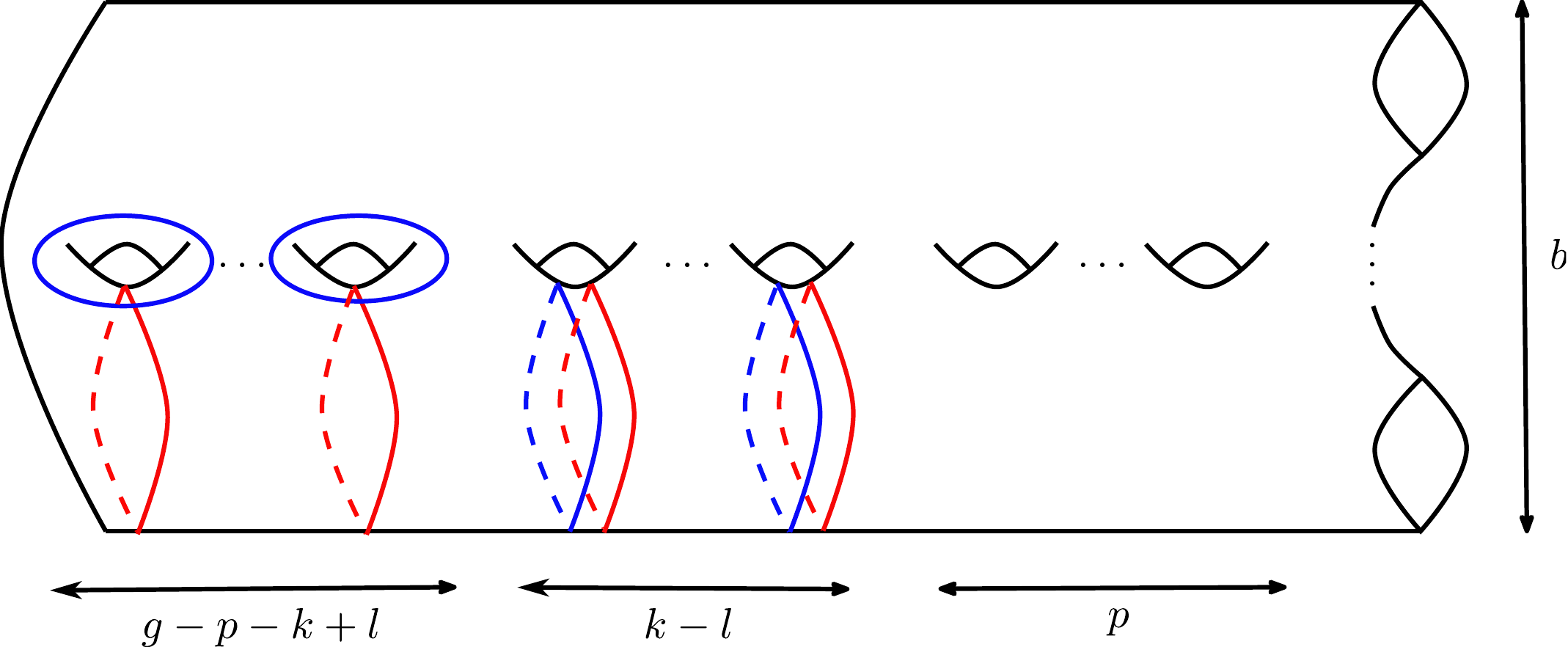}
\end{center}
\setlength{\captionmargin}{50pt}
\caption{The model diagram for relative trisection diagrams, where $l=2p+b-1$.}
\label{fig:modeldiagramofrelativetrisection}
\end{figure}

It is shown that a $(g,k;p,b)$-relative trisection induces an open book decomposition to the boundary of the ambient 4-manifold whose page is a surface of genus $p$ with $b$ boundary components \cite[Lemma 11]{MR3770114}. Also note unbalanced relative trisections and relative trisection diagrams can be considered as well as trisections of closed 4-manifolds.

An  \textit{arced relative trisection diagram} is roughly a relative trisection diagram with arcs \cite{MR4354420}.  An algorithm for drawing the arcs is developed. 

\begin{thm}[{\cite[Theorem 5]{MR3770114}}]
Let $(\Sigma, \alpha, \beta, \gamma)$ be a relative trisection diagram and $\Sigma_\alpha$ a surface obtained by surgerying $\Sigma$ along $\alpha$ curves. Then, three kinds of collections of arcs $a$, $b$ and $c$ are obtained as follows:

\begin{enumerate}
\item Draw a collection of arcs $a$ in $\Sigma$ so that cutting $\Sigma_\alpha$ along an image of  $a$ gives a disk.
\item Handle sliding a parallel copy of $a$ over $\alpha$ curves so that the copy does not intersect with $\beta$ curves. If necessary, $\beta$ curves can be handle slided over $\beta$ curves. The resulting collection of arcs is $b$. Let the $\beta$ at the end of this step be $\beta^{'}$.
\item Handle sliding a parallel copy of $b$ over $\beta^{'}$ curves so that the copy does not intersect with $\gamma$ curves. If necessary, $\gamma$ curves can be handle slided over $\gamma$ curves. The resulting collection of arcs is $c$. 
\end{enumerate}

\end{thm}

For a relative trisection $T_i$ of a 4-manifold $X_i$ ($i=1,2$), if there is a diffeomorphism between $\partial{X_1}$ and $\partial{X_2}$ compatible with their open book decompositions induced from $T_i$, we can obtain a trisection of $X_1 \cup X_2$ by gluing $T_1$ and $T_2$ \cite{MR3999550}. In particular, a trisection of the double can be obtained as follows:

\begin{thm}[{\cite[Corollary 2.8]{MR3999550}}]\label{thm:gluing}
If a 4-manifold $X$ with non-empty connected boundary admits a $(g,k;p,b)$-relative trisection, then the double of $X$ admits a $(2g+b-1,2k-\ell)$-trisection, where $\ell=2p+b-1$.
\end{thm}

Note that for a $(g;k_1,k_2,k_3;p,b)$-relative trisection, the type of the trisection is $(2g+b-1;2k_1-\ell,2k_2-\ell,2k_3-\ell)$.

Let $(\Sigma,\alpha,\beta,\gamma,a,b,c)$ be an arced relative trisection diagram of $X$ and $(\overline{\Sigma},\overline{\alpha},\overline{\beta},\overline{\gamma},\overline{a},\overline{b},\overline{c})$ the arced relative trisection diagram of $\overline{X}$ obtained by taking the mirror image of $(\Sigma,\alpha,\beta,\gamma,a,b,c)$. We define $\Sigma^{*}=\Sigma \cup \overline{\Sigma}$ and $\alpha^{*}= \alpha \cup \overline{\alpha} \cup (a \cup \overline{a})$ (we define $\beta^{*}$ and $\gamma^{*}$ in the same way).

\begin{thm}[{\cite[Corollary 2.13]{MR3999550}}]
The 4-tuple $(\Sigma^{*};\alpha^{*},\beta^{*},\gamma^{*})$ is a trisection diagram corresponding to the trisection in Theorem \ref{thm:gluing}. 
\end{thm}


\subsection{Mazur type 4-manifolds}
A contractible 4-manifold with boundary is called of \textit{Mazur type} if it admits a handle decomposition consisting of only single 0-, 1-, and 2-handle. Akbulut and Kirby \cite{MR0544597} introduced a Mazur type 4-manifold $W^{\pm}(\ell,k)$ depicted in Figure \ref{fig:AK}. This can be regarded as a generalization of the Akbulut cork since $W^{-}(0,0)$ is diffeomorphic to the Akbulut cork, which is the first example of corks found by Akbulut. The Mazur type $W^{\pm}(\ell,k)$ has the following properties:


\begin{prop}[{\cite[Proposition 1]{MR0544597}}]
For any integers $\ell$ and $k$, $W^{\pm}(\ell,k) \cong W^{\pm}(\ell+1,k-1)$ and $W^{-}(\ell,k) \cong \overline{W^{+}(-\ell,-k+3)}$.
\end{prop}

\begin{rem}
It is known that the double of each Mazur type is diffeomorphic to $S^4$ \cite{MR0125574}. In particular, we consider the double of $W^{-}(0,n+2)$ for any integer $n$ in Section \ref{sec:main}.
\end{rem}

\begin{figure}[h]
\begin{tabular}{cc}
\begin{minipage}{0.4\hsize}
\begin{center}
\includegraphics[width=8cm, height=3cm, keepaspectratio, scale=1]{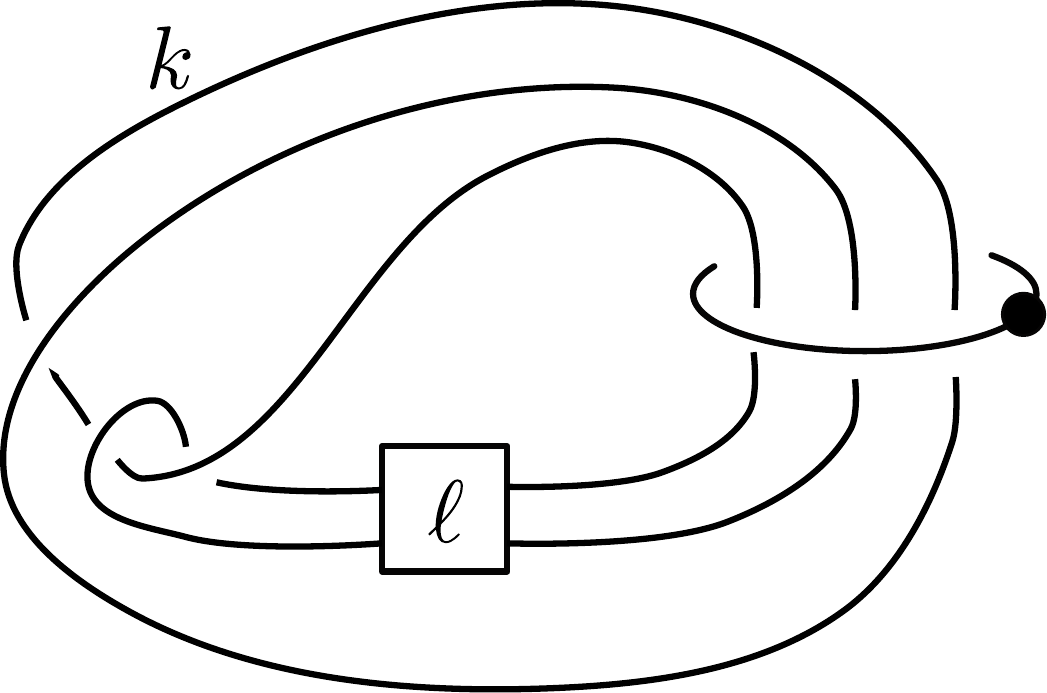}
\end{center}
\setlength{\captionmargin}{50pt}
\subcaption{$W^{-}(\ell,k)$}
\label{fig:}
\end{minipage} 
\begin{minipage}{0.4\hsize}
\begin{center}
\includegraphics[width=8cm, height=3cm, keepaspectratio, scale=1]{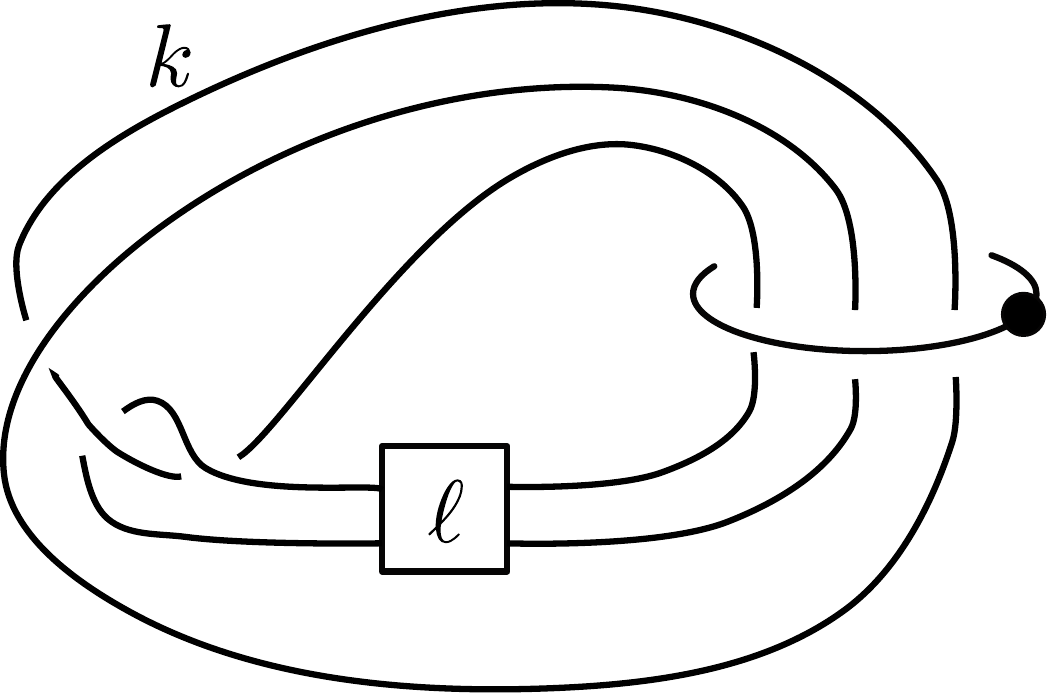}
\end{center}
\setlength{\captionmargin}{50pt}
\subcaption{$W^{+}(\ell,k)$}
\label{fig:}
\end{minipage} 
\end{tabular}
\setlength{\captionmargin}{50pt}
\caption{Kirby diagrams of the Mazur type $W^{\pm}(\ell,k)$ introduced by Akbulut and Kirby.}
\label{fig:AK}
\end{figure}

\subsection{An algorithm converting Kirby diagrams into trisection diagrams}\label{subsec:algorithm}
Kepplinger \cite{MR4460230} introduced an algorithm which takes Kirby diagrams to trisection diagrams. In this subsection, we recall a simplified version of the algorithm also introduced in \cite{MR4460230}, which can make the genus of the resulting trisection diagram smaller significantly.

\begin{thm}[{\cite[Algorithm 1 and Remark 2.2]{MR4460230}}]
Let $X$ be a closed orientable 4-manifold. Suppose that $X$ is described by a Kirby diagram with some 1-handles and a framed attaching link $L=K_1 \cup \dots \cup K_\ell$. One can obtain a trisection diagram of $X$ from this Kirby diagram by performing the following steps:

\begin{enumerate}
\item Convert each pair of attaching balls describing a 1-handle into a pair of disks and add a parallel pair of red and blue curves parallel to the boundary of the disk.
\item Add a +1- or (-1)-kink to an attaching circle which has no self-crossings.
\item Convert each crossing as in Figure \ref{fig:convert}. Then, one matches the framing of $K_i$ with the surface framing by winding around a handle (i.e. Dehn twisting on the red curve in Figure \ref{fig:convert}) that only $K_i$ runs. Let $g$ be the number of $\alpha$, or equivalently, $\beta$ curves at the end of this step.
\item Handle sliding $\alpha$ curves so that $\lvert \alpha_i \cap K_j \rvert = \delta_{ij}$ ($1 \le i \le g$, $1 \le j \le \ell$).
\item Convert $K_j$ to $\gamma_j$.
\item Draw $\gamma_i$ ($\ell+1 \le i \le g$) as a curve parallel to $\alpha_i$ which does not intersect with $K_j$ ($1 \le j \le \ell$).
\end{enumerate}
\end{thm}

\begin{figure}[h]
\begin{center}
\includegraphics[width=9cm, height=8cm, keepaspectratio, scale=1]{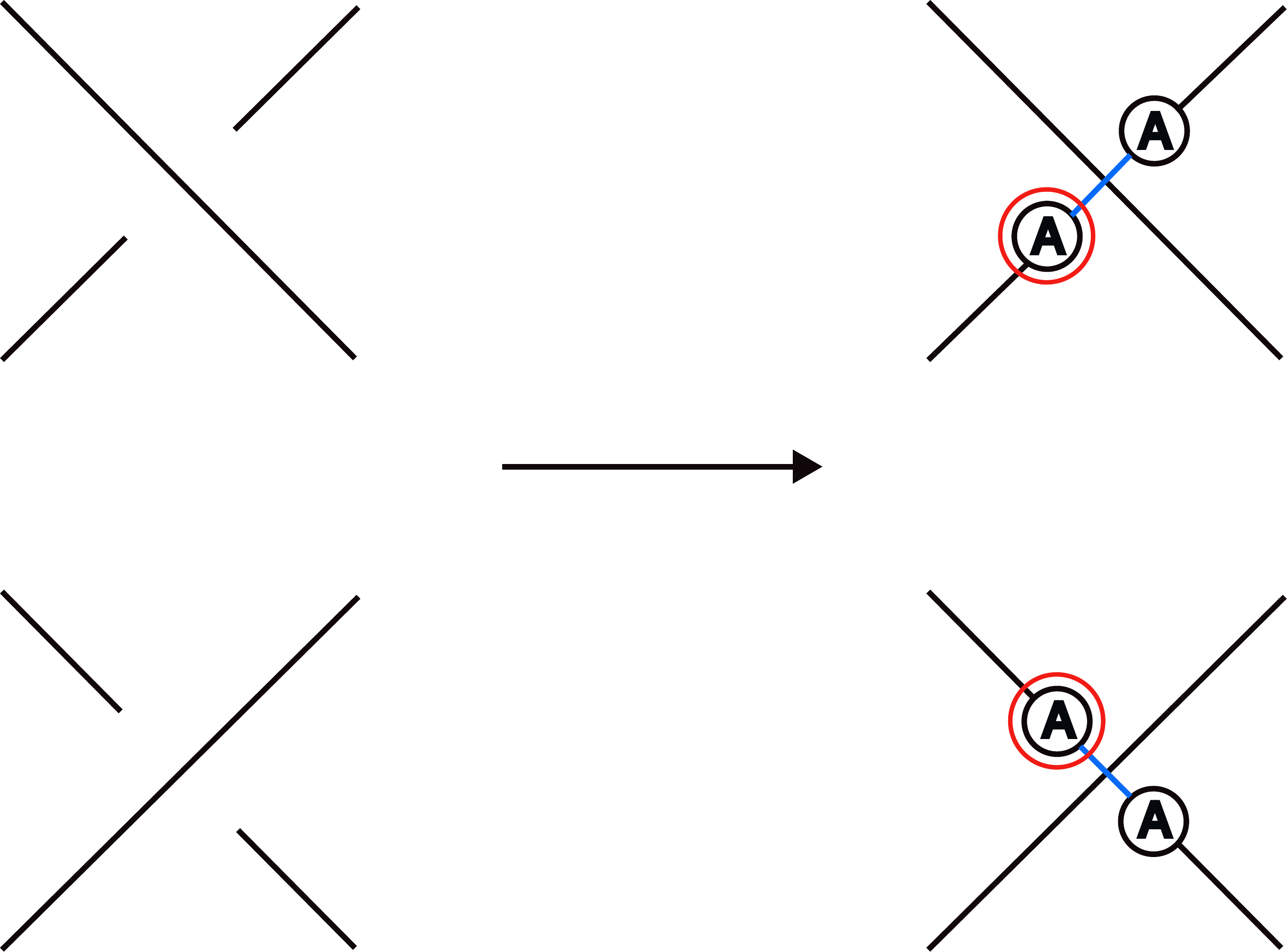}
\end{center}
\setlength{\captionmargin}{50pt}
\caption{At the step 3 in the algorithm, each crossing is converted as described in this figure.}
\label{fig:convert}
\end{figure}

\section{main theorems}\label{sec:main}
In this section, we construct two kinds of trisection diagrams of the double $DW^{-}(0,n+2)$ of $W^{-}(0,n+2)$ in a different way and show that they are standard for all $n$.

For any integer $n$, Takahashi \cite{takahashi2023exotic} constructed a $(3,3;0,4)$-relative trisection diagram of $W^{-}(0,n+2)$. The diagram is depicted in Figure \ref{fig:Takahashi}. Note that the right handed Dehn twist along a simple closed curve $c$ is denoted by $t_c$.

\begin{figure}[h]
\begin{center}
\includegraphics[width=10cm, height=8cm, keepaspectratio, scale=1]{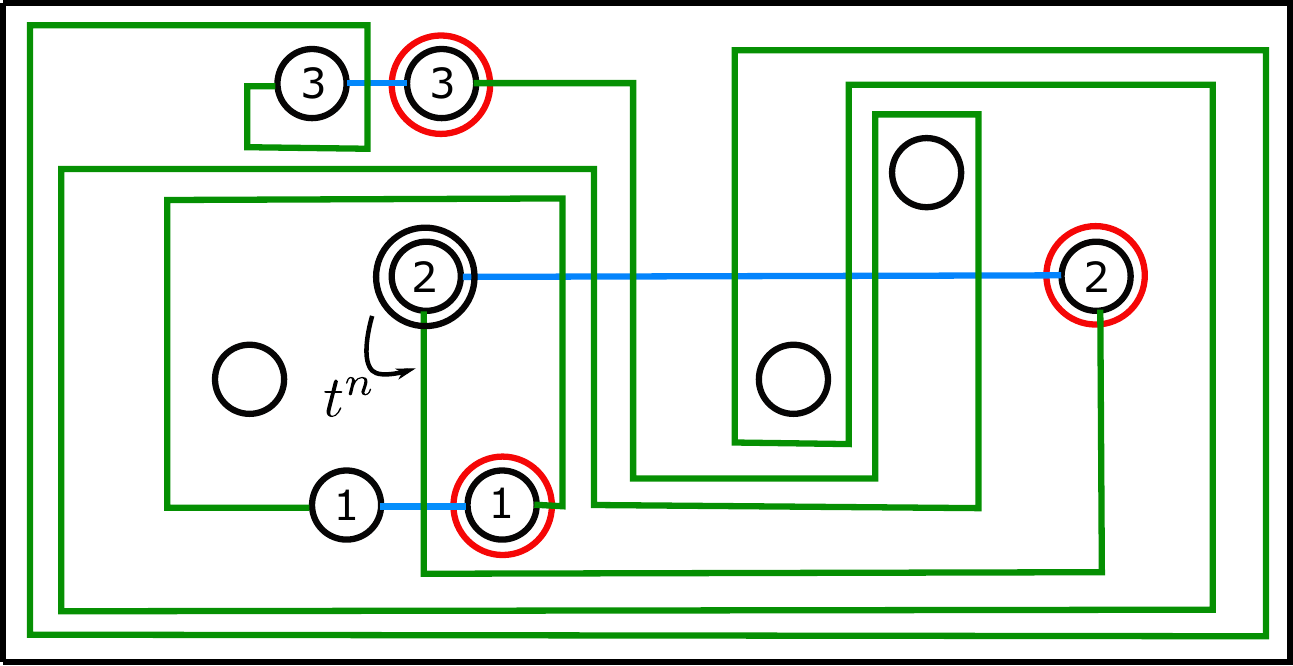}
\end{center}
\setlength{\captionmargin}{50pt}
\caption{The $(3,3;0,4)$-relative trisection diagram $\mathcal{D}_n$ of $W^{-}(0,n+2)$. The arrow labeled $t^n$ describes the $n$ times Dehn twists for the green curve intersecting with the black curve along it. This notation is used throughout the proofs of the main theorems.}
\label{fig:Takahashi}
\end{figure}

\begin{lem}
An arced relative trisection diagram of $W^{-}(0,n+2)$ is depicted in Figure \ref{fig:artd}.
\end{lem}

\begin{proof}
We use an algorithm in \cite{MR3770114} to depict the arcs. Note that handle sliding a $\gamma$ curve and the black curve is performed. We can perform such a handle slide for the black curve due to a certain commutative property between handle slides and Dehn twists. See Figure 9 in \cite{isoshima2023infinitely} for details.
\end{proof}

\begin{figure}[h]
\begin{center}
\includegraphics[width=13cm, height=20cm, keepaspectratio, scale=1]{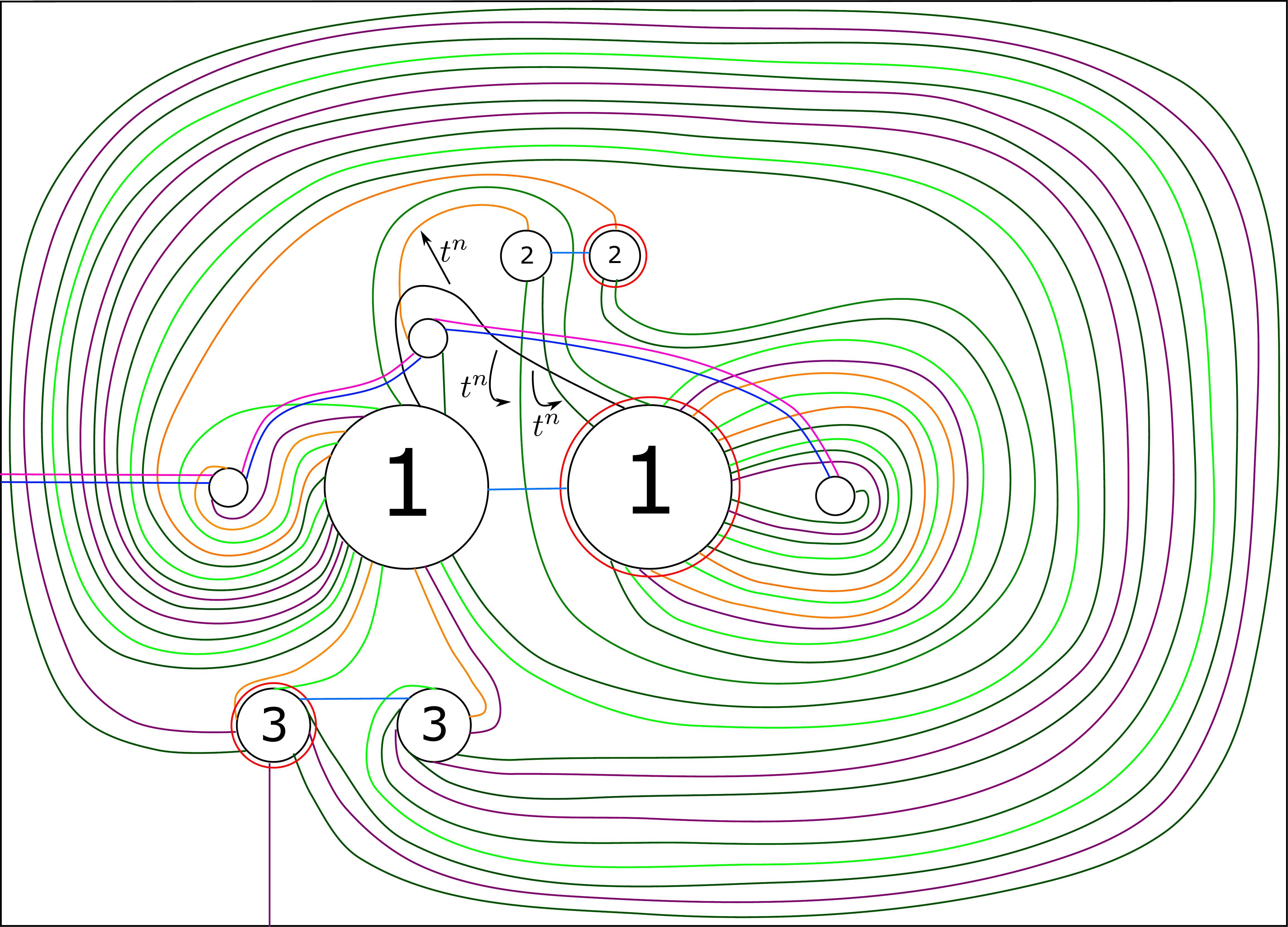}
\end{center}
\setlength{\captionmargin}{50pt}
\caption{An arced relative trisection diagram of $\mathcal{D}_n$.}
\label{fig:artd}
\end{figure}

We can construct a trisection diagram of $DW^{-}(0,n+2)$ from this arced relative trisection diagram by using \cite[Corollary 2.13]{MR3999550}. Figure \ref{fig:DDn} describes the trisection diagram 

\begin{figure}[h]
\begin{center}
\includegraphics[width=10cm, height=10cm, keepaspectratio, scale=1]{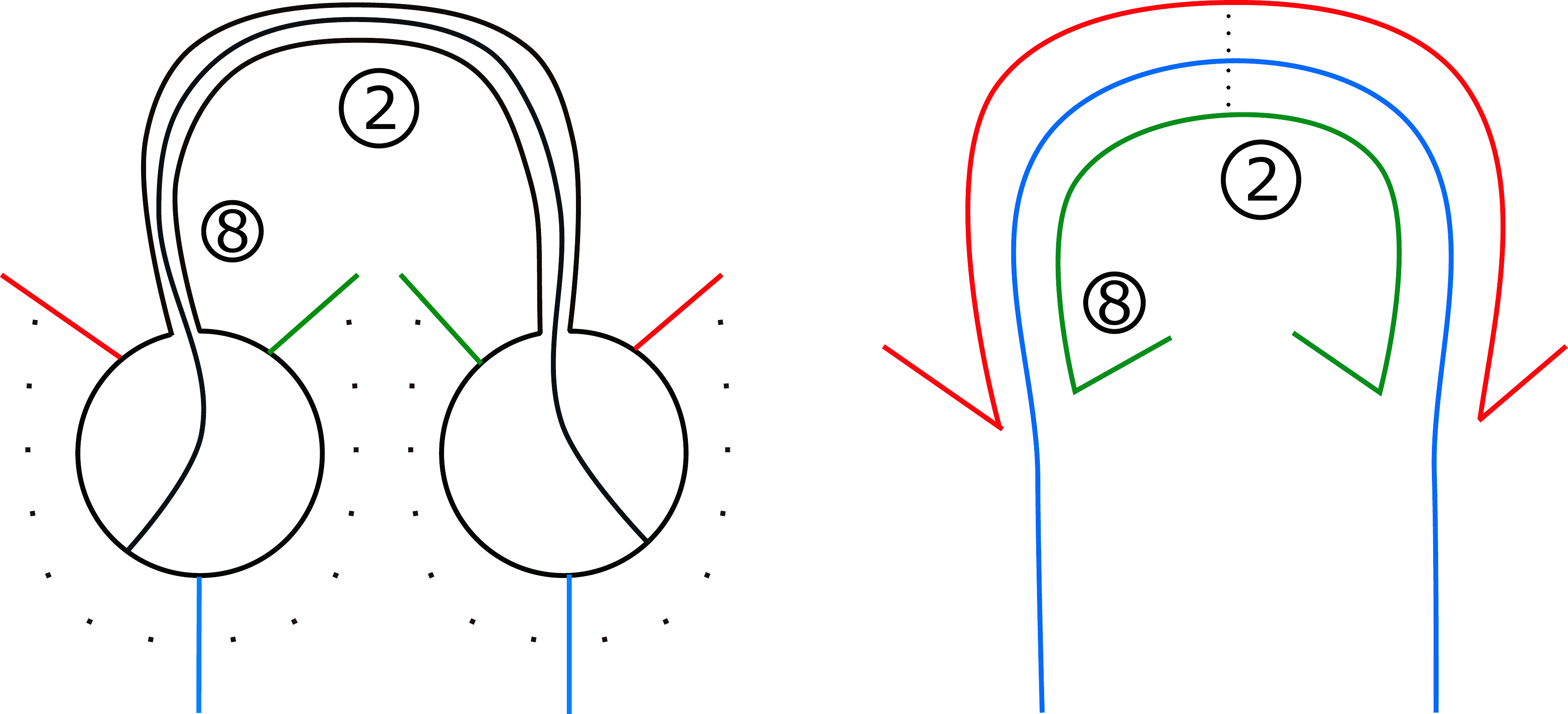}
\end{center}
\setlength{\captionmargin}{50pt}
\caption{The left figure describes the right one in Figures \ref{fig:after1}, \ref{fig:before2_1} and \ref{fig:before2}.}
\label{fig:identify}
\end{figure}

\begin{figure}[h]
\begin{center}
\includegraphics[width=10cm, height=10cm, keepaspectratio, scale=1]{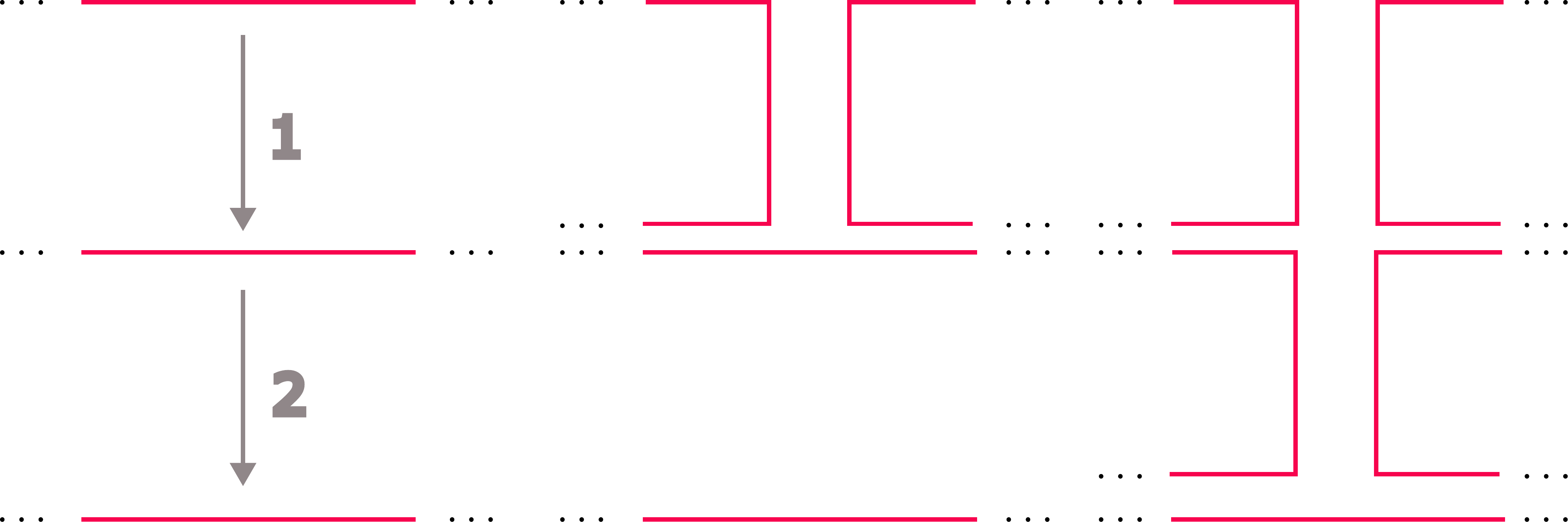}
\end{center}
\setlength{\captionmargin}{50pt}
\caption{In the proofs of our theorems, the most left figure means that firstly we perform the handle slide labeled 1 (the central figure), and then we perform the handle slide labeled 2 (the most right figure).}
\label{fig:handleslide}
\end{figure}

\begin{figure}[h]
\begin{center}
\includegraphics[width=13cm, height=20cm, scale=1]{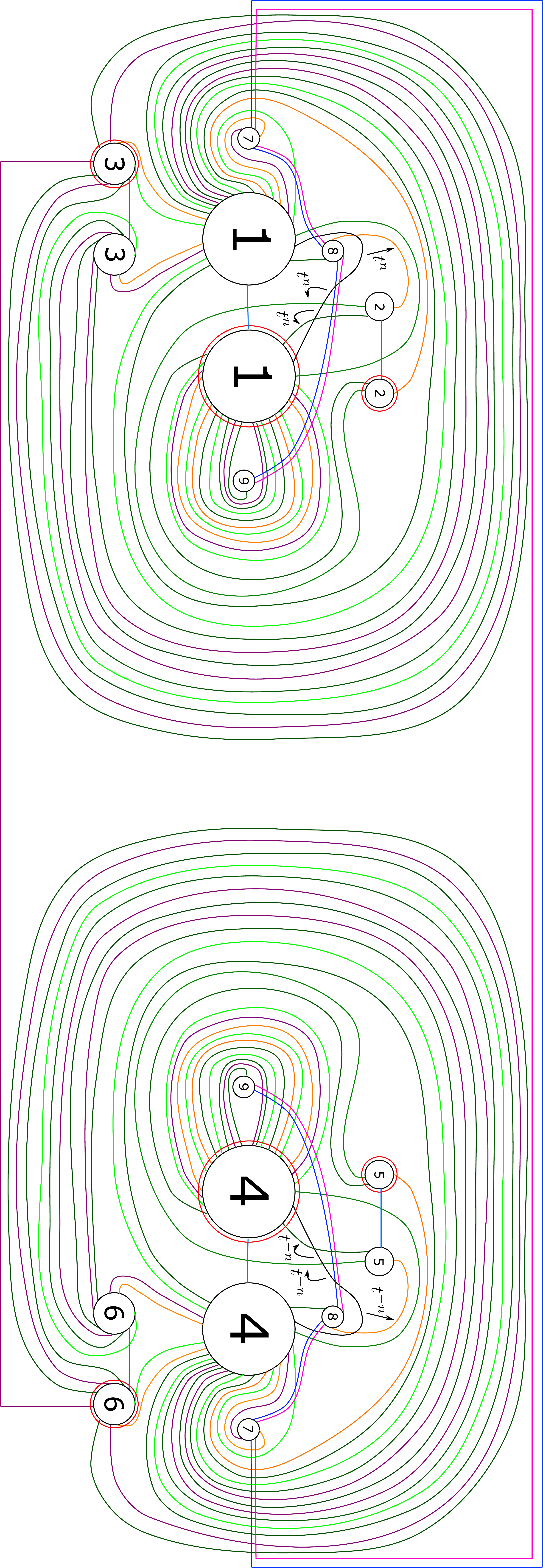}
\end{center}
\setlength{\captionmargin}{50pt}
\caption{A $(9,3)$-trisection diagram of $DW^{-}(0,n+2)$ obtained by doubling the arced relative trisection diagram of $\mathcal{D}_n$. The red and pink curves are $\alpha$ curves, and the light and dark blue curves are $\beta$ curves. The other color curves are $\gamma$ curves. We use several colors for the $\gamma$ curves for visual clarity.}
\label{fig:DDn}
\end{figure}

\begin{thm}\label{thm1}
The trisection diagram depicted in Figure \ref{fig:DDn} is standard for any integer $n$.
\end{thm}

\begin{proof}
We show the statement by destabilizing the trisection diagram continuously. Destabilizations in the proof are based on \cite[Lemma 8]{MR4480889}. See Figure \ref{fig:handleslide} on handle slides in the proof.

\subsection*{The first destabilization}
By performing the handle slides depicted in Figure \ref{fig:start} in gray, Figure \ref{fig:before1} is obtained. Then, by destabilizing Figure \ref{fig:before1}, we have Figure \ref{fig:after1}.

\begin{figure}[h]
\begin{center}
\includegraphics[width=13cm, height=20cm, scale=1]{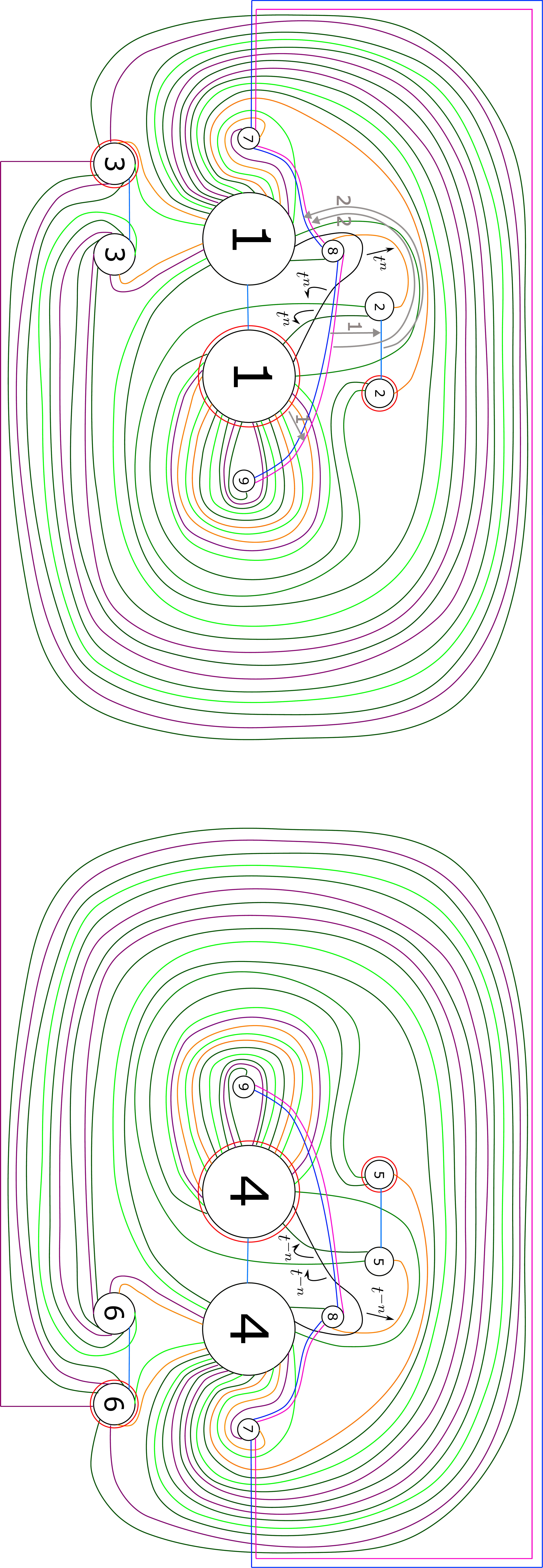}
\end{center}
\setlength{\captionmargin}{50pt}
\caption{The starting diagram in the proof.}
\label{fig:start}
\end{figure}

\begin{figure}[h]
\begin{center}
\includegraphics[width=13cm, height=20cm, scale=1]{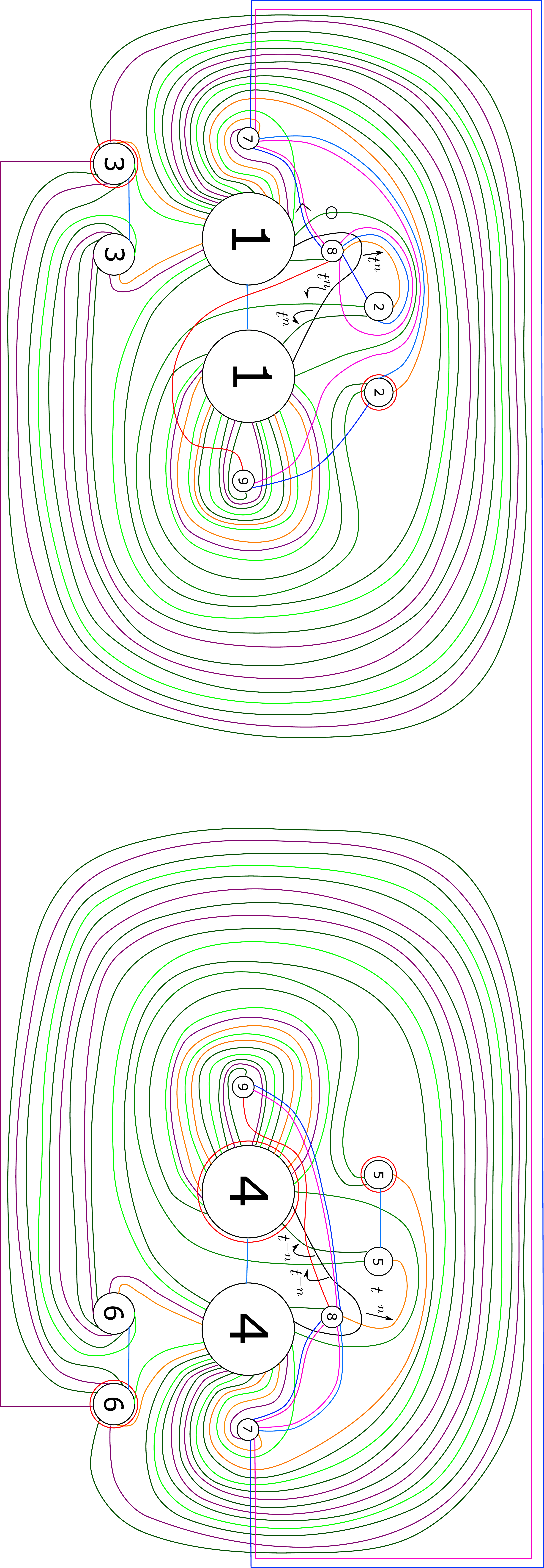}
\end{center}
\setlength{\captionmargin}{50pt}
\caption{The trisection diagram obtained from Figure \ref{fig:start} by some handle slides (before the first destabilization). Destabilize this diagram for three curves with the arrow and circle.}
\label{fig:before1}
\end{figure}

\begin{figure}[h]
\begin{center}
\includegraphics[width=13cm, height=20cm, scale=1]{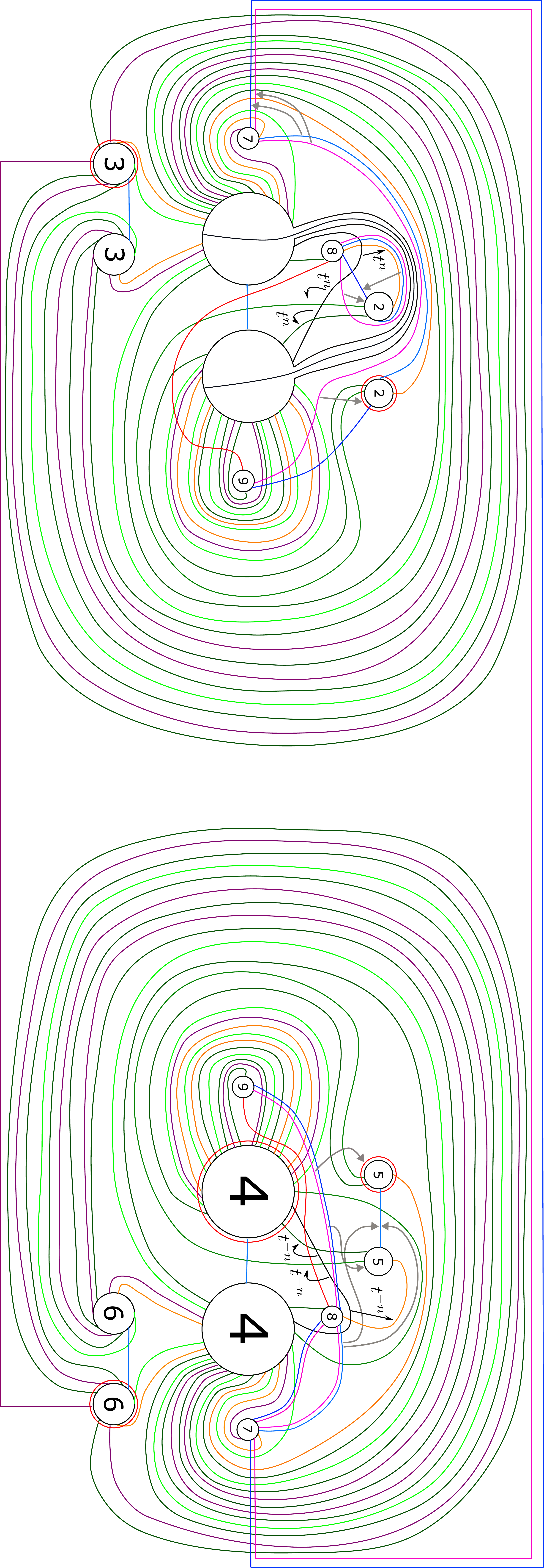}
\end{center}
\setlength{\captionmargin}{50pt}
\caption{After the first destabilization.}
\label{fig:after1}
\end{figure}

\subsection*{The second destabilization}
By performing the handle slides depicted in Figure \ref{fig:after1} in gray, Figure \ref{fig:before2_1} is obtained (see Figure \ref{fig:identify} for the notation). Then, perform $t^{-n}_c$, where $c$ is the black meridian curve in the disk labeled 2. Moreover, by performing the handle slides depicted in Figure \ref{fig:before2_1} in gray, Figure \ref{fig:before2} is obtained. Then, by destabilizing Figure \ref{fig:before2}, we have Figure \ref{fig:after2}.

\begin{figure}[h]
\begin{center}
\includegraphics[width=13cm, height=20cm, scale=1]{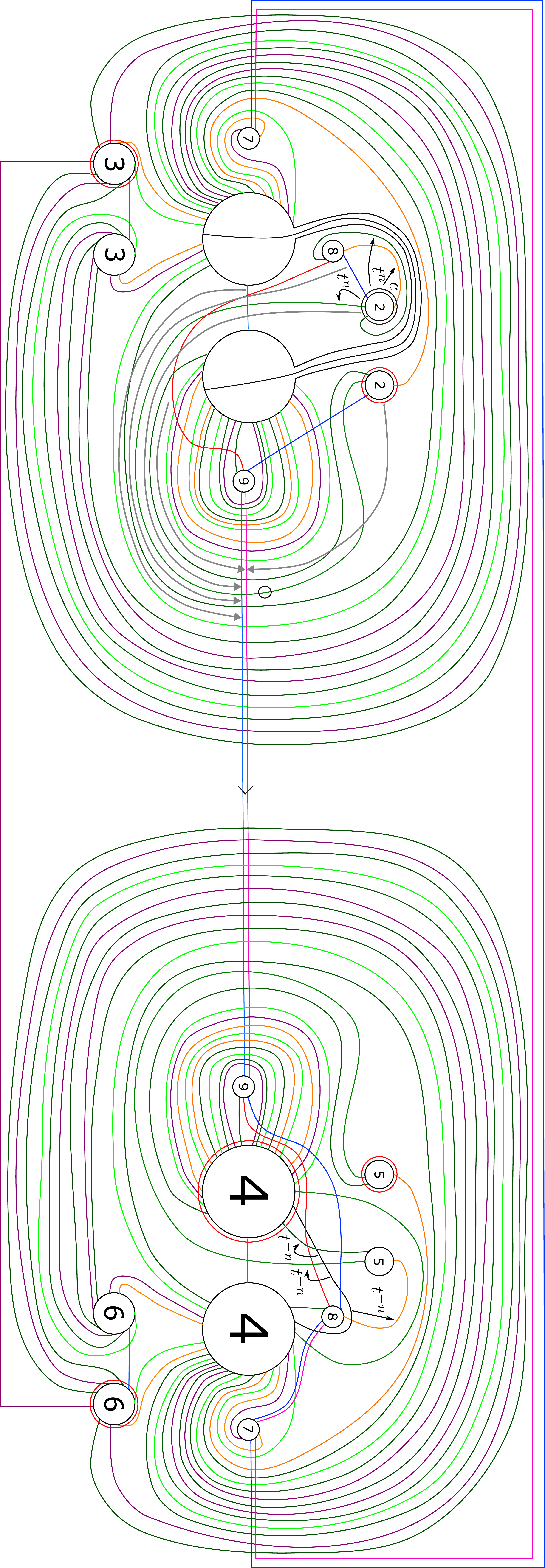}
\end{center}
\setlength{\captionmargin}{50pt}
\caption{The diagram obtained from Figure \ref{fig:after1} by some handle slides. This diagram can be destabilized for three curves with the arrow and circle.}
\label{fig:before2_1}
\end{figure}

\begin{figure}[h]
\begin{center}
\includegraphics[width=13cm, height=20cm, scale=1]{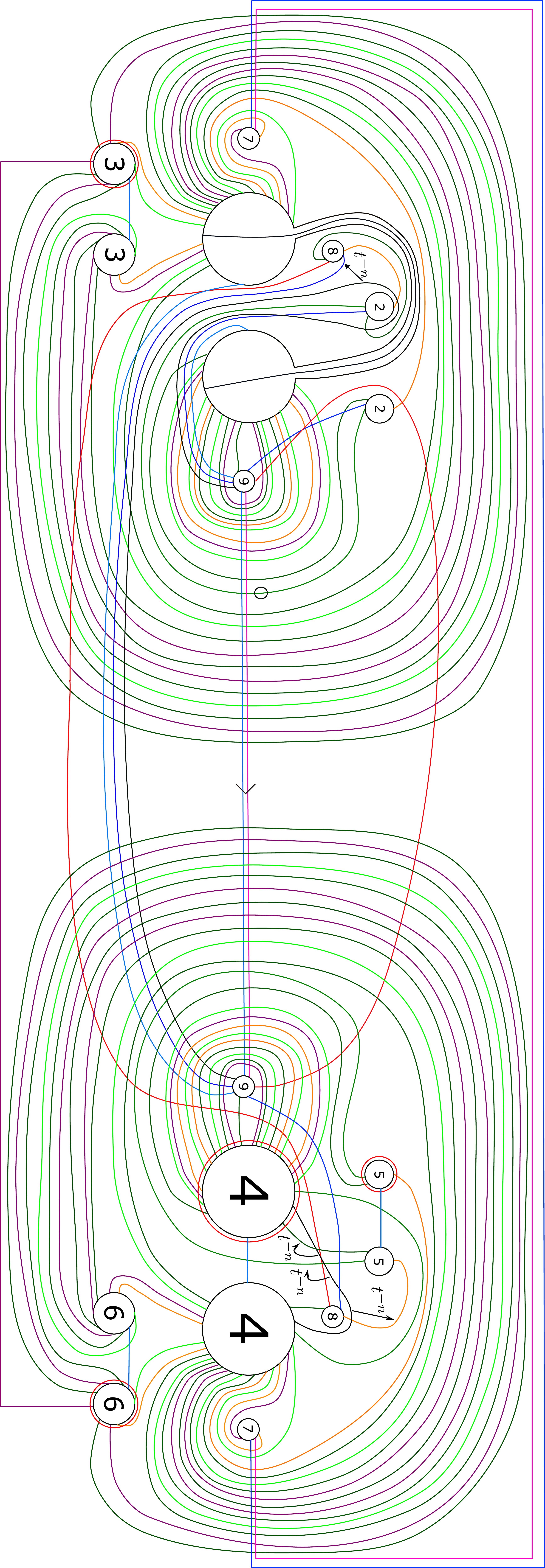}
\end{center}
\setlength{\captionmargin}{50pt}
\caption{The diagram obtained from Figure \ref{fig:before2_1} by some Dehn twist and handle slides (before the second destabilization).}
\label{fig:before2}
\end{figure}

\begin{figure}[h]
\begin{center}
\includegraphics[width=13cm, height=20cm, scale=1]{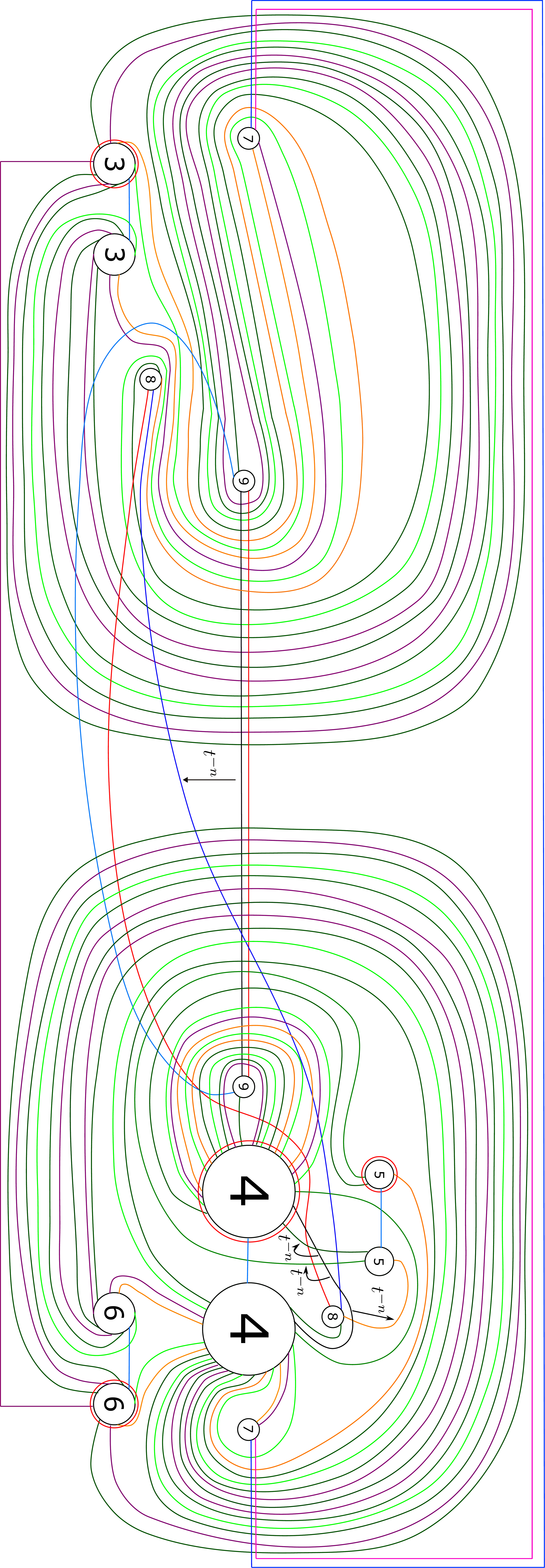}
\end{center}
\setlength{\captionmargin}{50pt}
\caption{After the second destabilization.}
\label{fig:after2}
\end{figure}

\subsection*{The third destabilization}
In Figure \ref{fig:after2}, drag the disk labeled 7. Then, Figure \ref{fig:before3_1} is obtained. In Figure \ref{fig:before3_1}, perform the handle slides indicated in gray and $t_{c^{'}}^{n}$. After that, by handle sliding $c^{'}$ over the $\beta$ curve on the disk labeled 4, $c$ and $c^{'}$ are parallel on the disk labeled 8. Since they act on the same $\beta$ curve on the opposite sign, the Dehn twists can be cancelled for the $\beta$ curves. Thus, we have Figure \ref{fig:before3_2} ($t_{c^{'}}^n$ acts only the $\alpha$ curve). In Figure \ref{fig:before3_2}, perform the handle slides indicated in gray, and then $t_{c_1}$ and handle slide $\beta$ curves on the $\beta$ curve on the disk labeled 3. After that, handle slide $\alpha$ curves on the $\alpha$ curve on the disk labeled 6 as mentioned above, and then perform $t_{c_2}^{-1}$. Then, Figure \ref{fig:before3_3} is obtained. In Figure \ref{fig:before3_3}, by performing the handle slides indicated in gray in order, we have Figure \ref{fig:before3_4}. By destabilizing Figure \ref{fig:before3_4} along the three curves labeled the arrow and circle, Figure \ref{fig:after3} is obtained.

\begin{figure}[h]
\begin{center}
\includegraphics[width=13cm, height=20cm, scale=1]{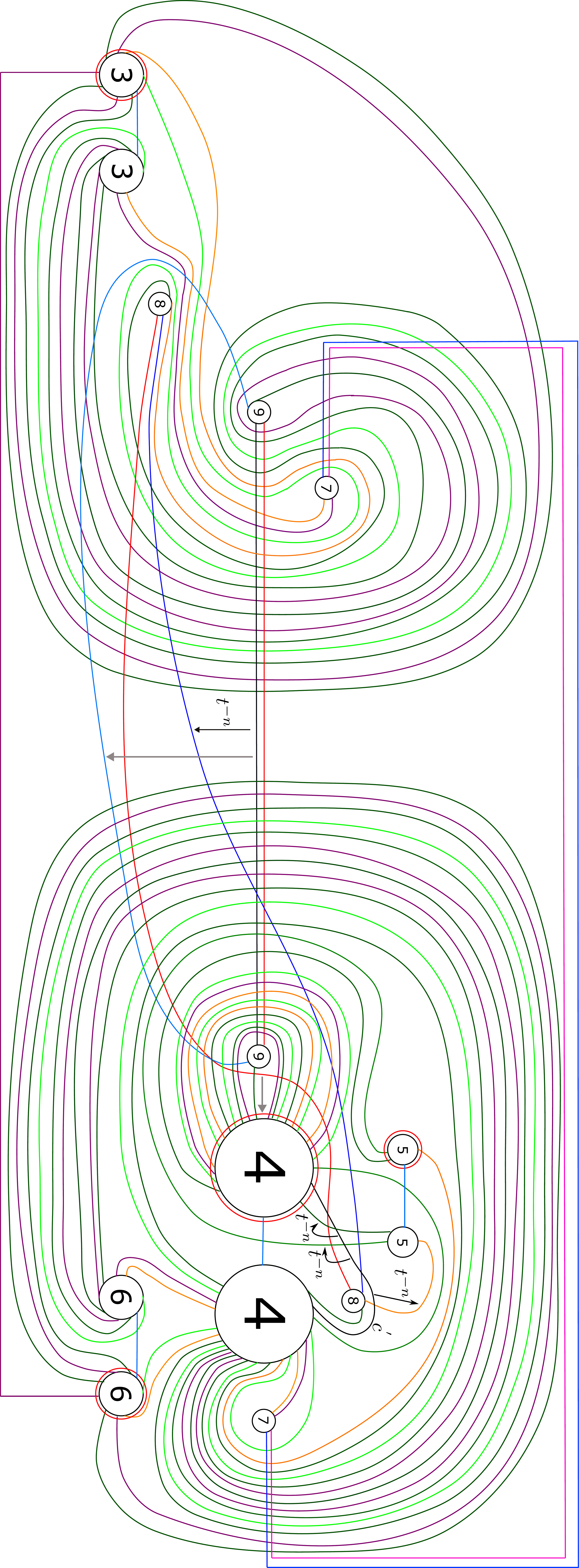}
\end{center}
\setlength{\captionmargin}{50pt}
\caption{Drag the left disk labeled 7.}
\label{fig:before3_1}
\end{figure}

\begin{figure}[h]
\begin{center}
\includegraphics[width=13cm, height=20cm, scale=1]{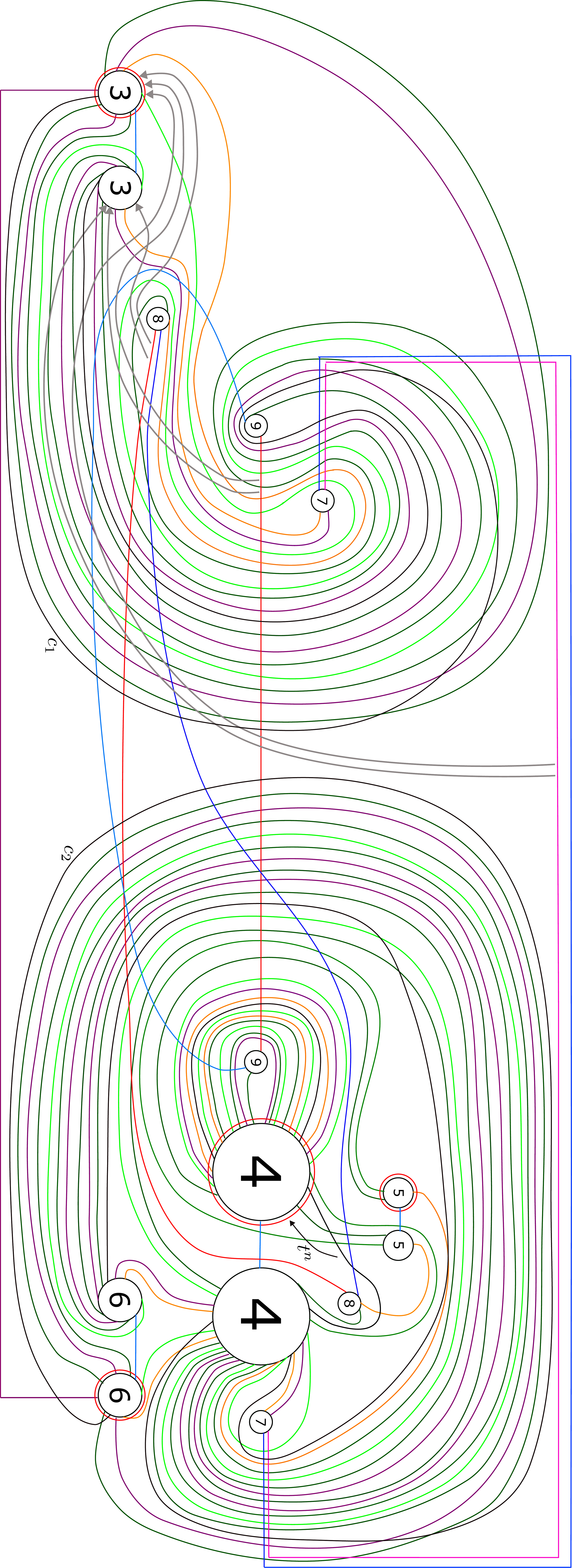}
\end{center}
\setlength{\captionmargin}{50pt}
\caption{The diagram obtained from Figure \ref{fig:before3_1} by some Dehn twists and handle slides.}
\label{fig:before3_2}
\end{figure}

\begin{figure}[h]
\begin{center}
\includegraphics[width=13cm, height=20cm, scale=1]{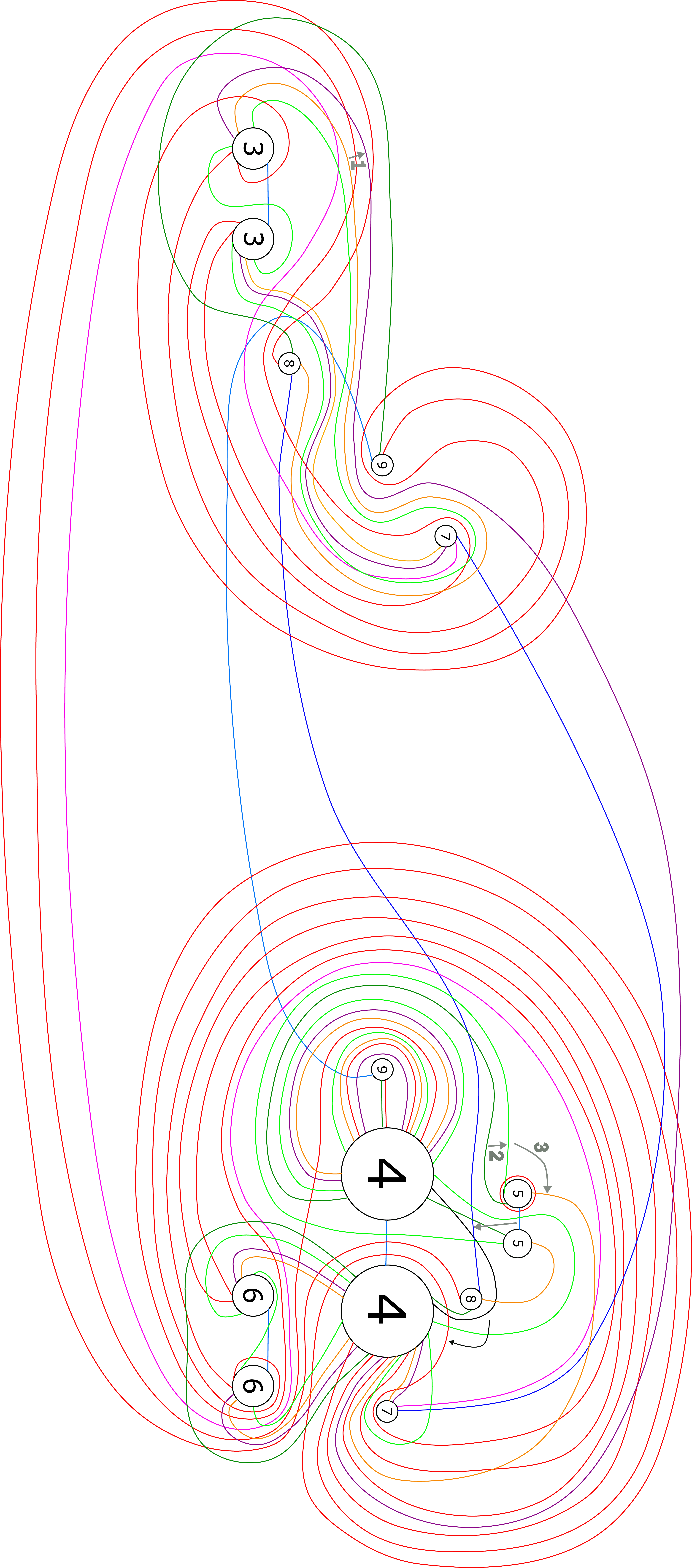}
\end{center}
\setlength{\captionmargin}{50pt}
\caption{The diagram obtained from Figure \ref{fig:before3_2} by some Dehn twists and handle slides.}
\label{fig:before3_3}
\end{figure}

\begin{figure}[h]
\begin{center}
\includegraphics[width=13cm, height=20cm, scale=1]{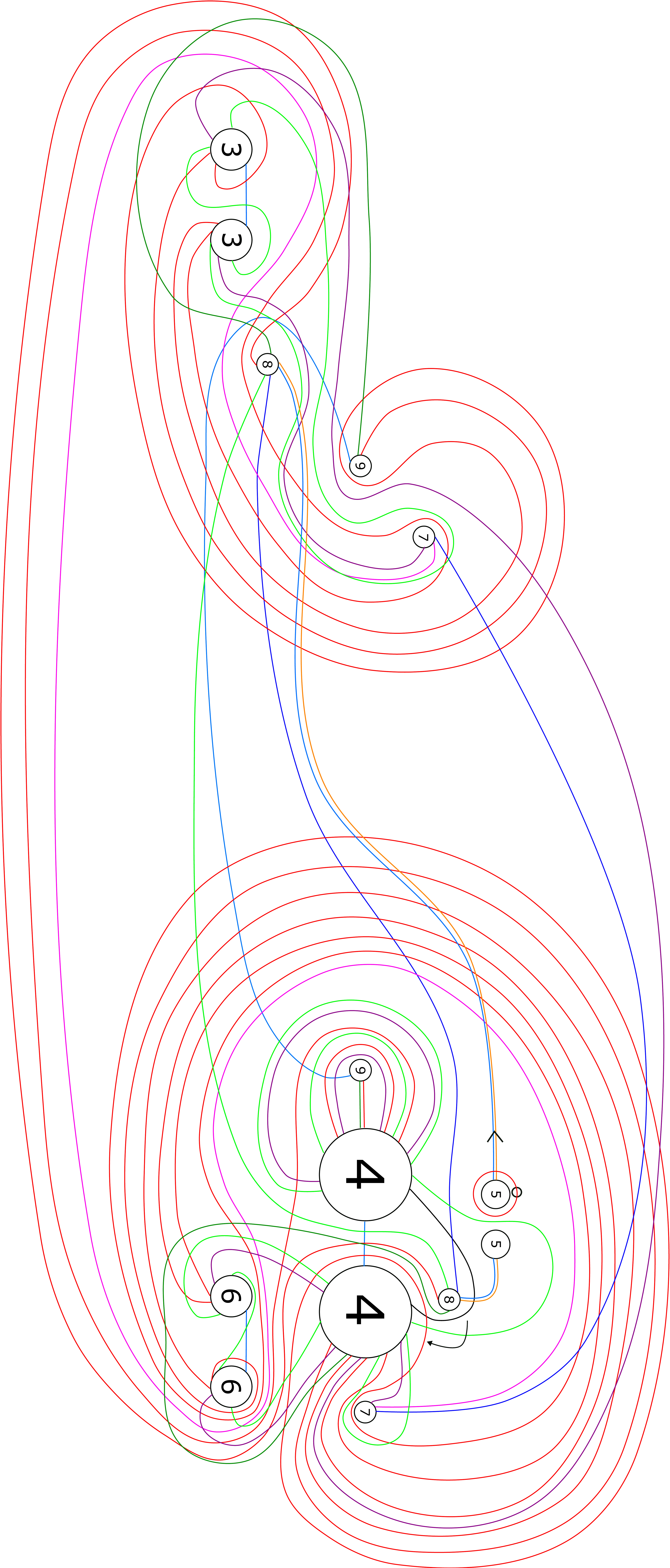}
\end{center}
\setlength{\captionmargin}{50pt}
\caption{The diagram obtained from Figure \ref{fig:before3_3} by some handle slides. (before the third destabilization)}
\label{fig:before3_4}
\end{figure}

\begin{figure}[h]
\begin{center}
\includegraphics[width=13cm, height=20cm, scale=1]{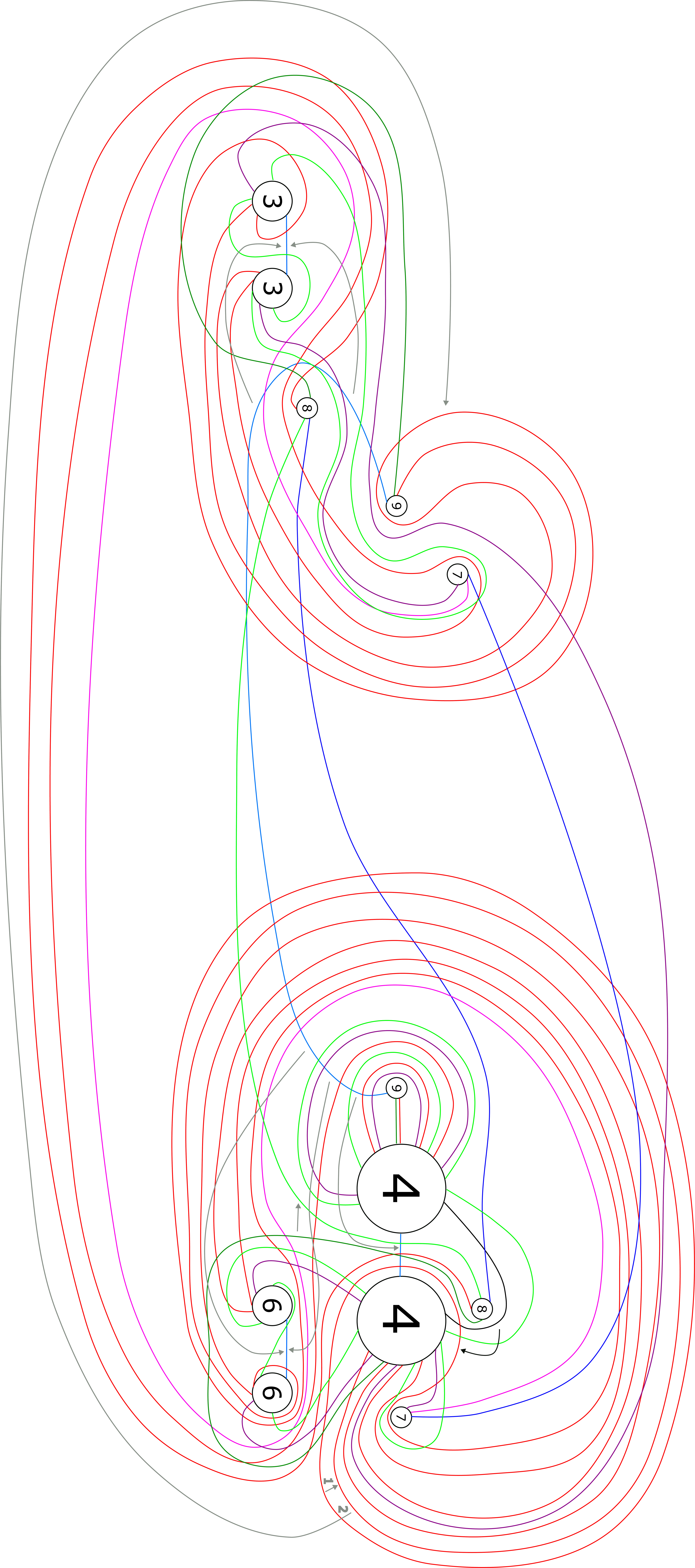}
\end{center}
\setlength{\captionmargin}{50pt}
\caption{After the third destabilization.}
\label{fig:after3}
\end{figure}

\subsection*{The fourth destabilization}
In Figure \ref{fig:after3}, by performing the handle slides indicated in gray in order, Figure \ref{fig:before4} is obtained. By destabilizing Figure \ref{fig:before4} along the three curves labeled the arrow and circle, Figure \ref{fig:after4} is obtained.

\begin{figure}[h]
\begin{center}
\includegraphics[width=13cm, height=20cm, scale=1]{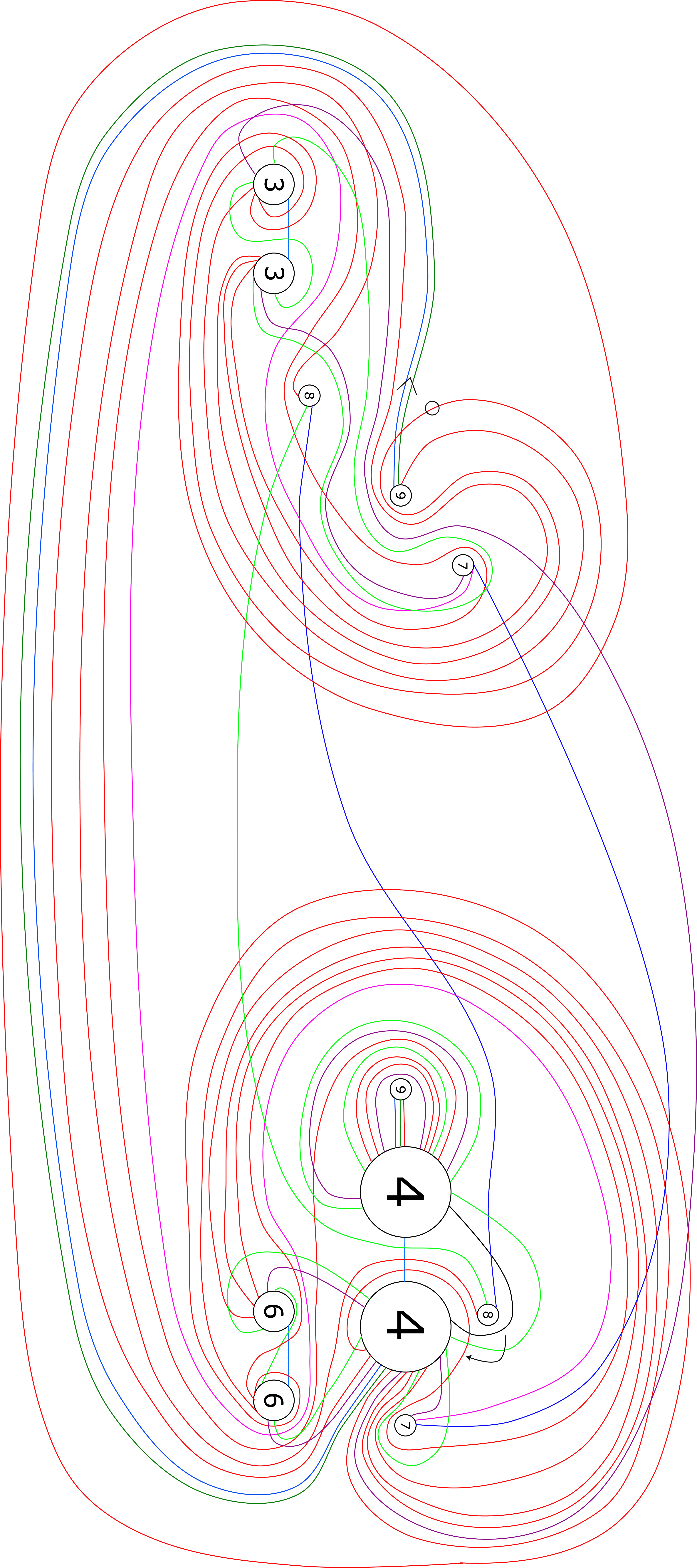}
\end{center}
\setlength{\captionmargin}{50pt}
\caption{The diagram obtained from Figure \ref{fig:after3} by some handle slides. (before the fourth destabilization)}
\label{fig:before4}
\end{figure}

\begin{figure}[h]
\begin{center}
\includegraphics[width=13cm, height=20cm, scale=1]{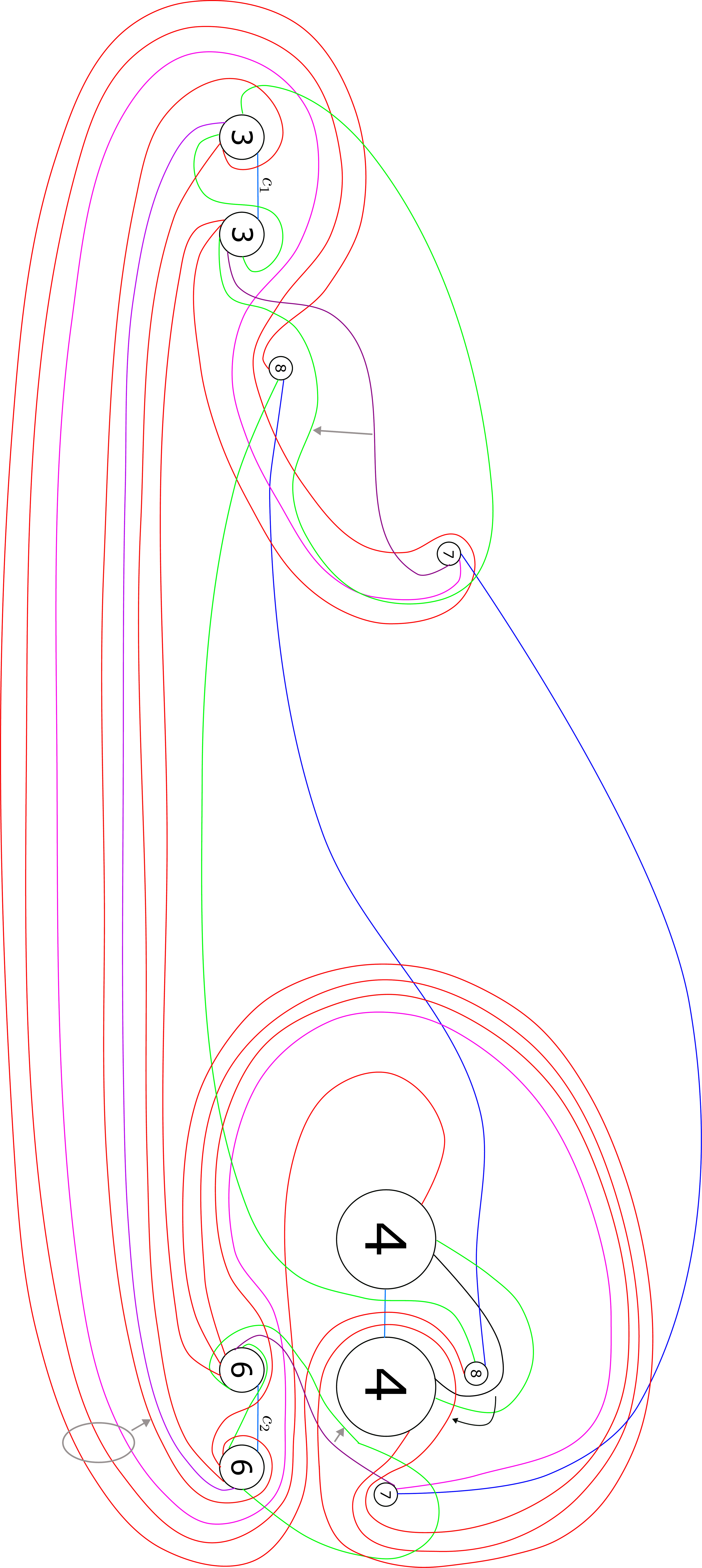}
\end{center}
\setlength{\captionmargin}{50pt}
\caption{After the fourth destabilization.}
\label{fig:after4}
\end{figure}

\subsection*{The fifth destabilization}
In Figure \ref{fig:after4}, perform $t_{c_2} \circ t_{c_1}^{-1}$, and then perform the handle slides indicated in gray. Then, Figure \ref{fig:before5_1} is obtained. In Figure \ref{fig:before5_1}, perform the handle slides indicated in gray, and then $t_c$. Then, Figure \ref{fig:before5_2} is obtained. In Figure \ref{fig:before5_2}, by performing the handle slides indicated in gray, we have Figure \ref{fig:before5_3}. In Figure \ref{fig:before5_3}, by performing the handle slides indicated in gray, we have Figure \ref{fig:before5_4}. By destabilizing Figure \ref{fig:before5_4} along the three curves labeled the arrow and circle, Figure \ref{fig:after5} is obtained.

\begin{figure}[h]
\begin{center}
\includegraphics[width=13cm, height=20cm, scale=1]{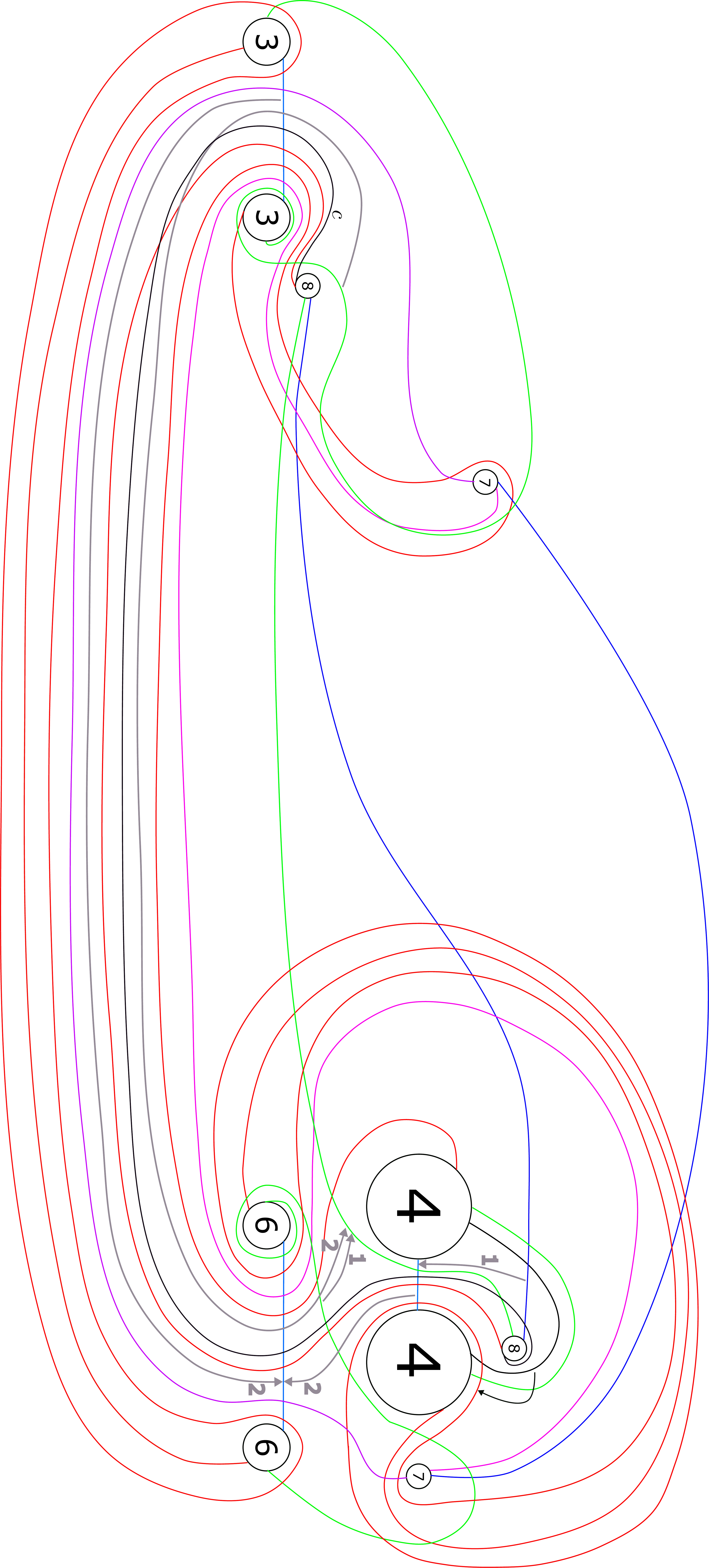}
\end{center}
\setlength{\captionmargin}{50pt}
\caption{The diagram obtained from Figure \ref{fig:after4} by some Dehn twists and handle slides.}
\label{fig:before5_1}
\end{figure}

\begin{figure}[h]
\begin{center}
\includegraphics[width=13cm, height=20cm, scale=1]{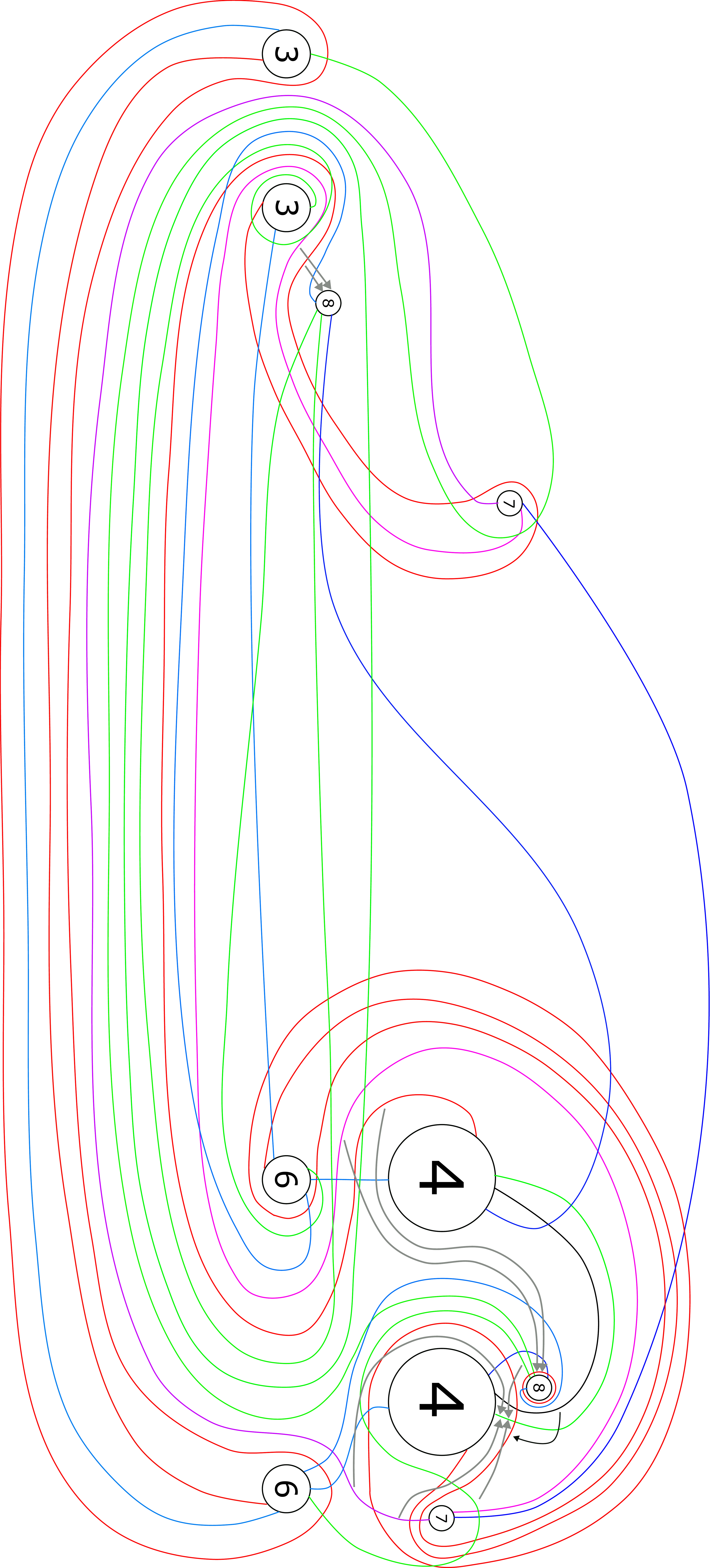}
\end{center}
\setlength{\captionmargin}{50pt}
\caption{The diagram obtained from Figure \ref{fig:before5_1} by some Dehn twists and handle slides.}
\label{fig:before5_2}
\end{figure}

\begin{figure}[h]
\begin{center}
\includegraphics[width=13cm, height=20cm, scale=1]{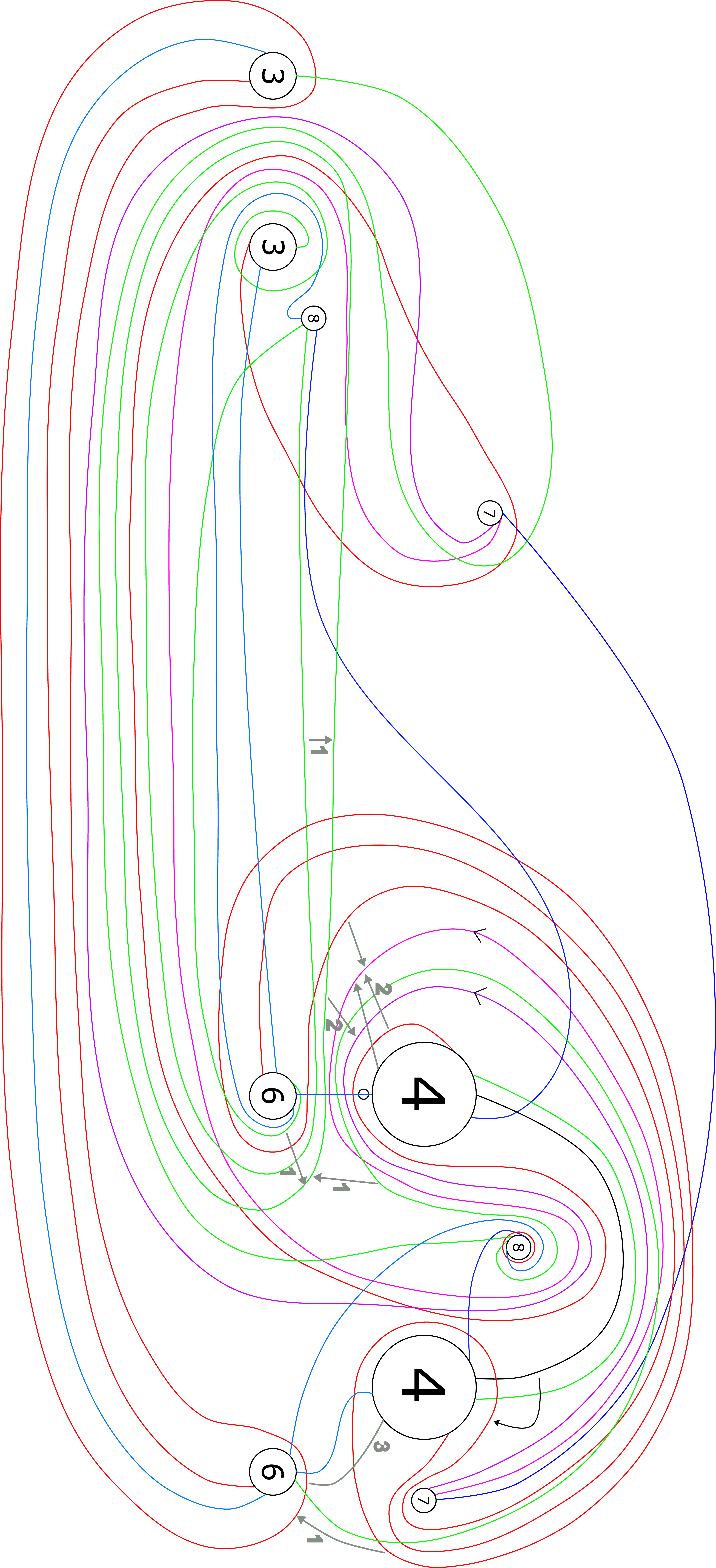}
\end{center}
\setlength{\captionmargin}{50pt}
\caption{The diagram obtained from Figure \ref{fig:before5_2} by some handle slides.}
\label{fig:before5_3}
\end{figure}

\begin{figure}[h]
\begin{center}
\includegraphics[width=13cm, height=20cm, scale=1]{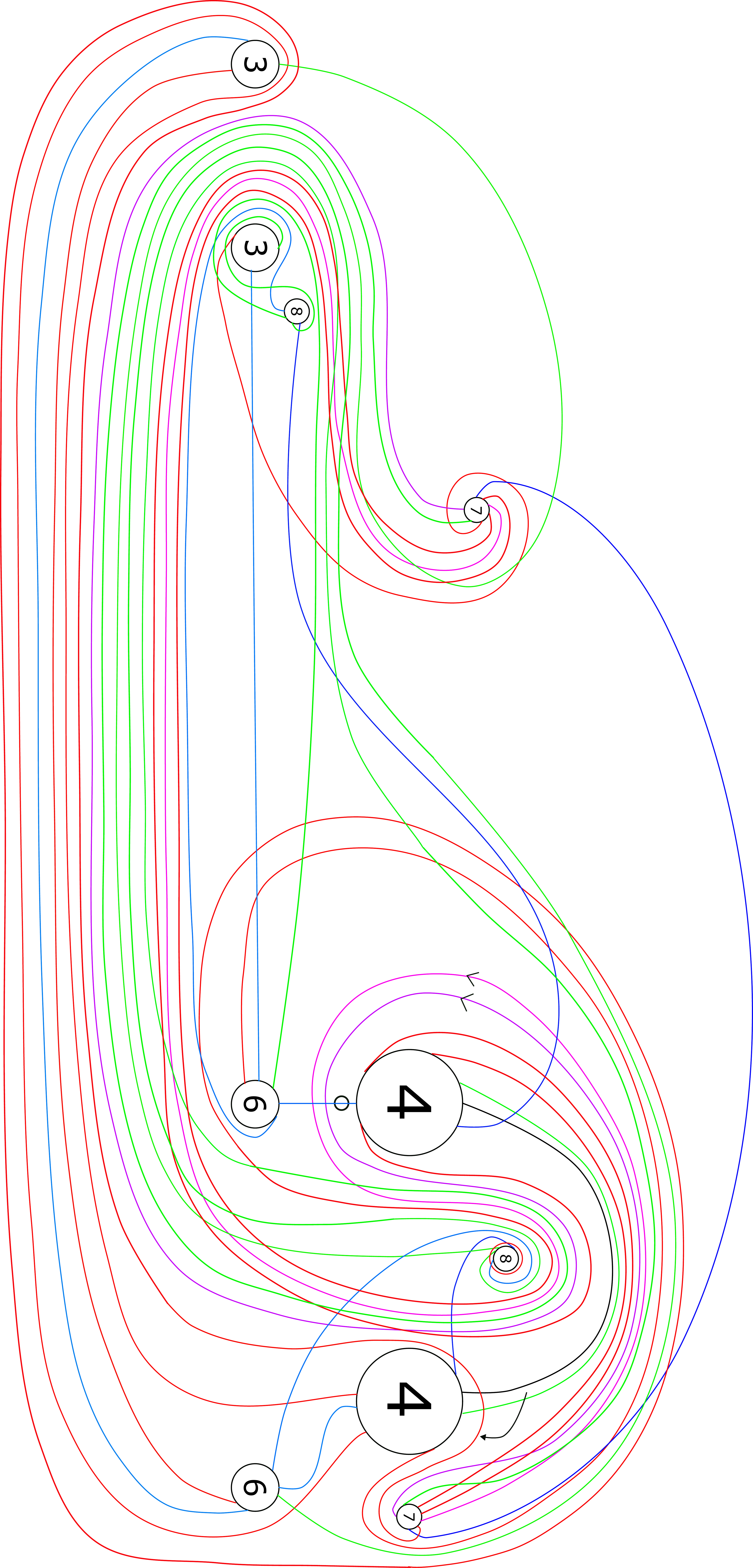}
\end{center}
\setlength{\captionmargin}{50pt}
\caption{The diagram obtained from Figure \ref{fig:before5_3} by some handle slides. (before the fifth destabilization)}
\label{fig:before5_4}
\end{figure}

\begin{figure}[h]
\begin{center}
\includegraphics[width=13cm, height=20cm, scale=1]{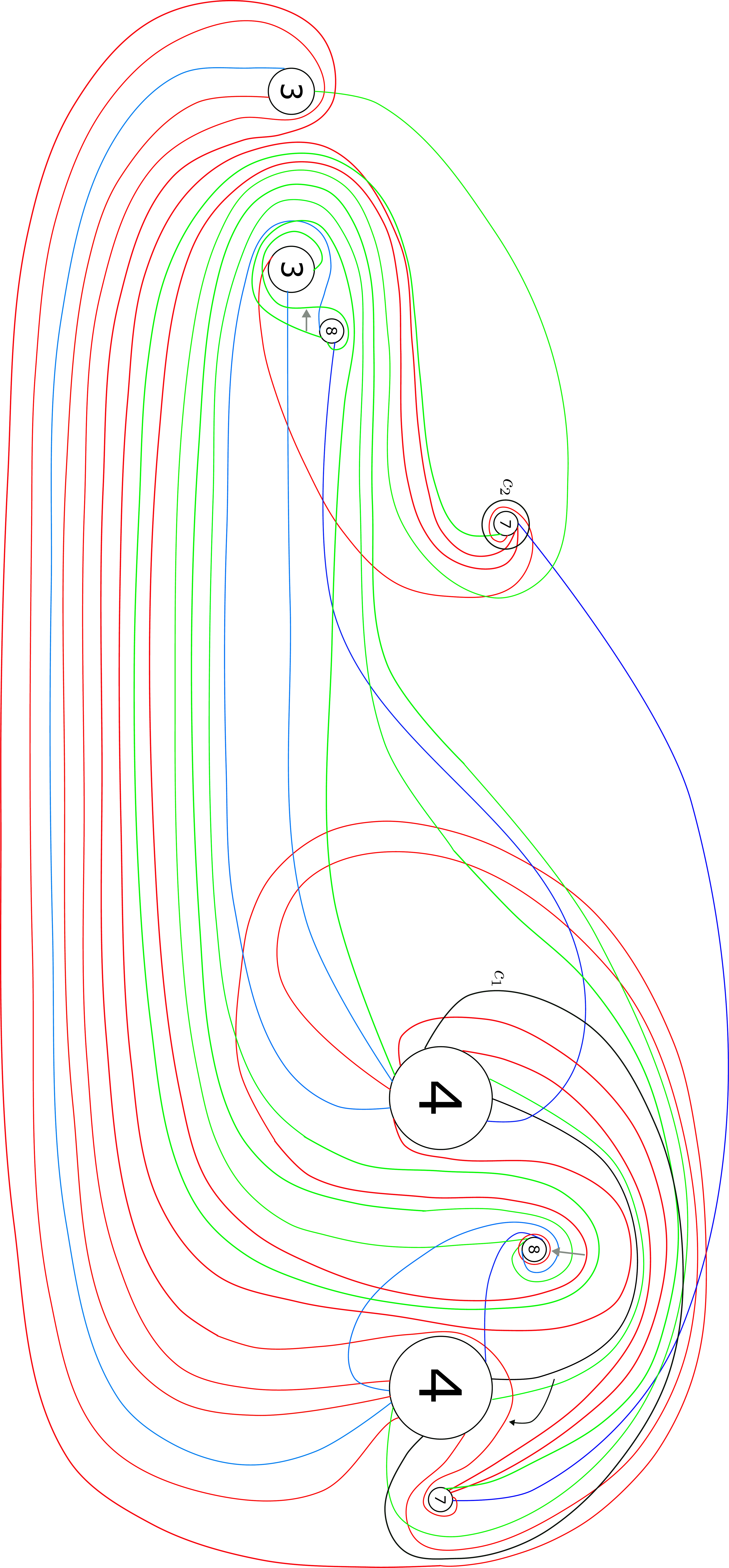}
\end{center}
\setlength{\captionmargin}{50pt}
\caption{After the fifth destabilization.}
\label{fig:after5}
\end{figure}

\subsection*{The sixth destabilization}
In Figure \ref{fig:after5}, perform the handle slides indicated in gray, and then $t_{c_2}^{-1} \circ t_{c_1}$. Then, Figure \ref{fig:before6_1} is obtained. In Figure \ref{fig:before6_1}, perform the handle slides indicated in gray, and then $t_{c_4}^{-1} \circ t_{c_3} \circ t_{c_2} \circ t_{c_1}^{-1}$. Then, Figure \ref{fig:before6_2} is obtained. In Figure \ref{fig:before6_2}, perform the handle slides indicated in gray, and then $t_c$. After that, by handle sliding $\alpha$ and $\gamma$ curves on the $\alpha$ curve on the disk labeled 8 and the $\gamma$ curve on the disk labeled 4, respectively, we have Figure \ref{fig:before6_3}. In Figure \ref{fig:before6_3}, by performing the handle slides indicated in gray, Figure \ref{fig:before6_4} is obtained. By destabilizing Figure \ref{fig:before6_4} along the three curves labeled the arrow and circle, we have Figure \ref{fig:after6}.

\begin{figure}[h]
\begin{center}
\includegraphics[width=13cm, height=20cm, scale=1]{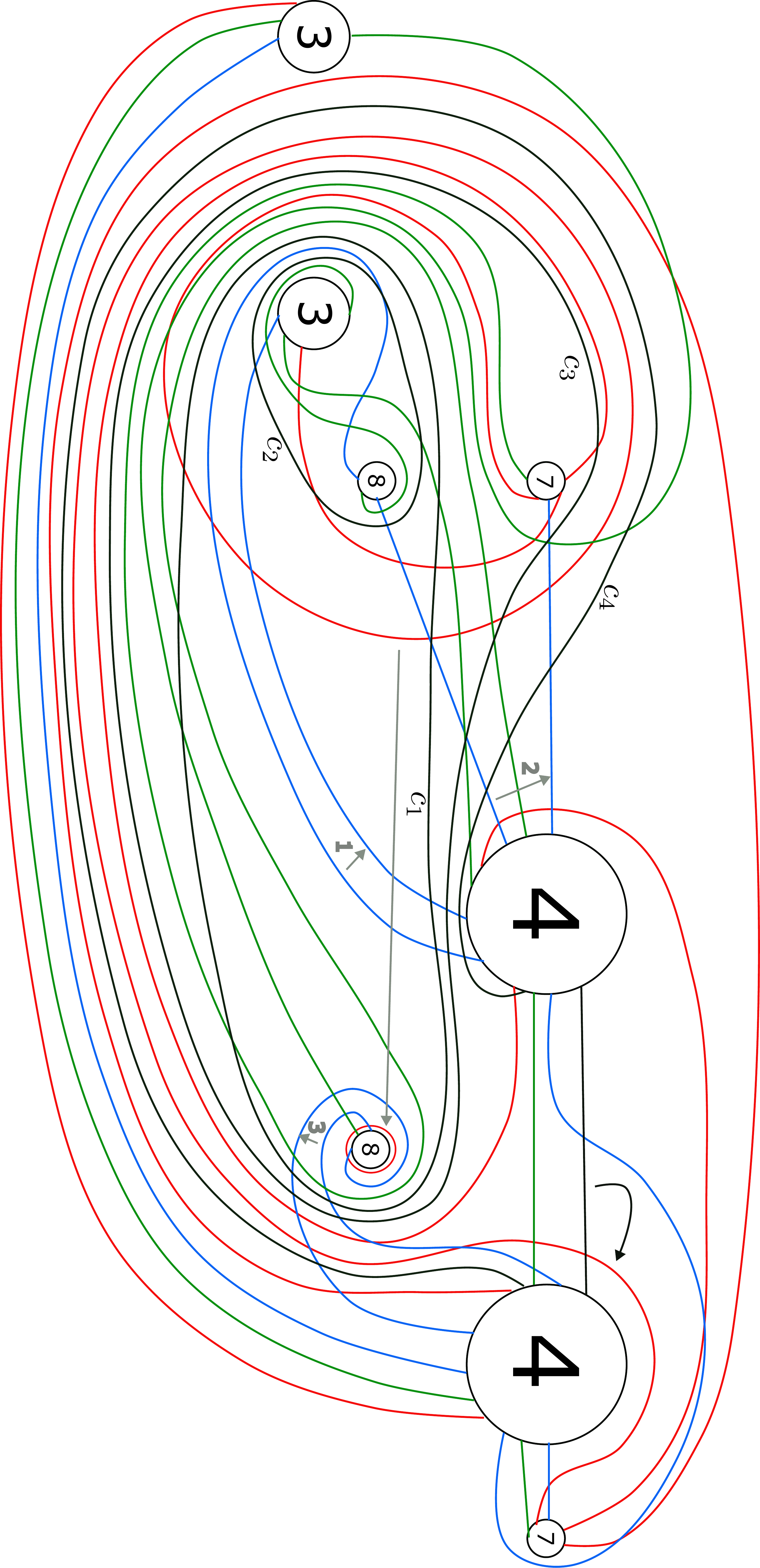}
\end{center}
\setlength{\captionmargin}{50pt}
\caption{The diagram obtained from Figure \ref{fig:after5} by some Dehn twists and handle slides.}
\label{fig:before6_1}
\end{figure}

\begin{figure}[h]
\begin{center}
\includegraphics[width=13cm, height=20cm, scale=1]{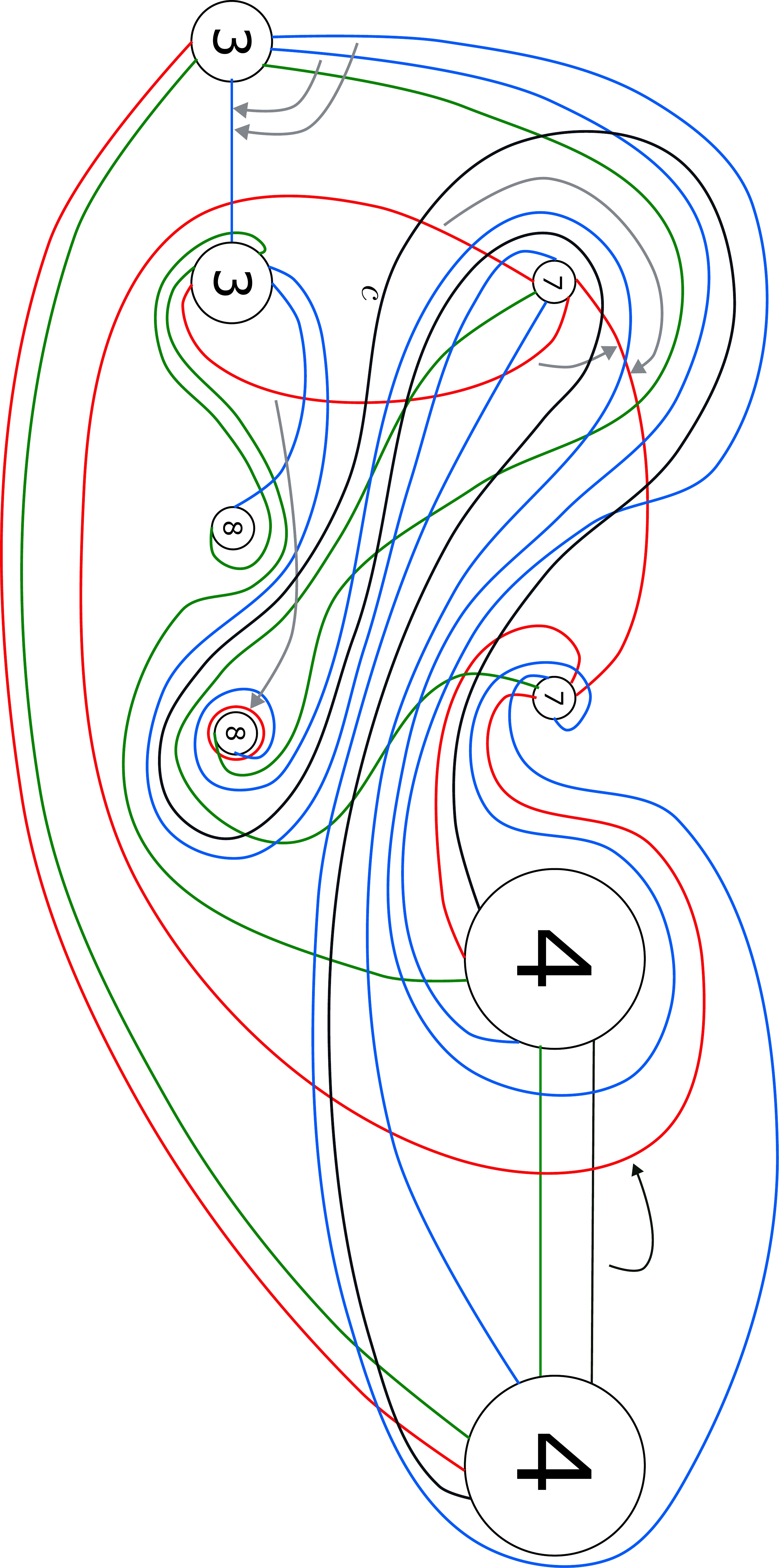}
\end{center}
\setlength{\captionmargin}{50pt}
\caption{The diagram obtained from Figure \ref{fig:before6_1} by some Dehn twists and handle slides.}
\label{fig:before6_2}
\end{figure}

\begin{figure}[h]
\begin{center}
\includegraphics[width=13cm, height=20cm, keepaspectratio, scale=1]{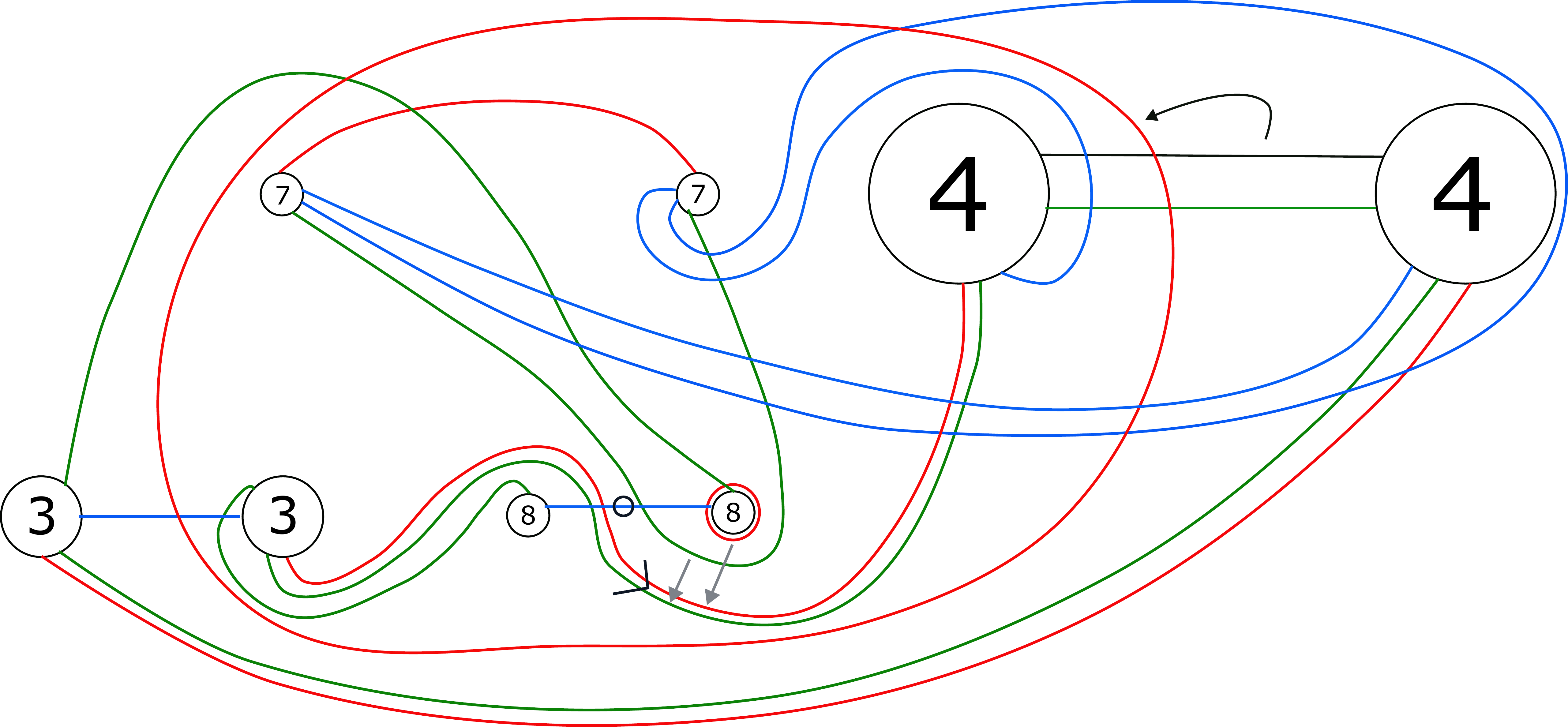}
\end{center}
\setlength{\captionmargin}{50pt}
\caption{The diagram obtained from Figure \ref{fig:before6_2} by some Dehn twists and handle slides.}
\label{fig:before6_3}
\end{figure}

\begin{figure}[h]
\begin{center}
\includegraphics[width=13cm, height=20cm, keepaspectratio, scale=1]{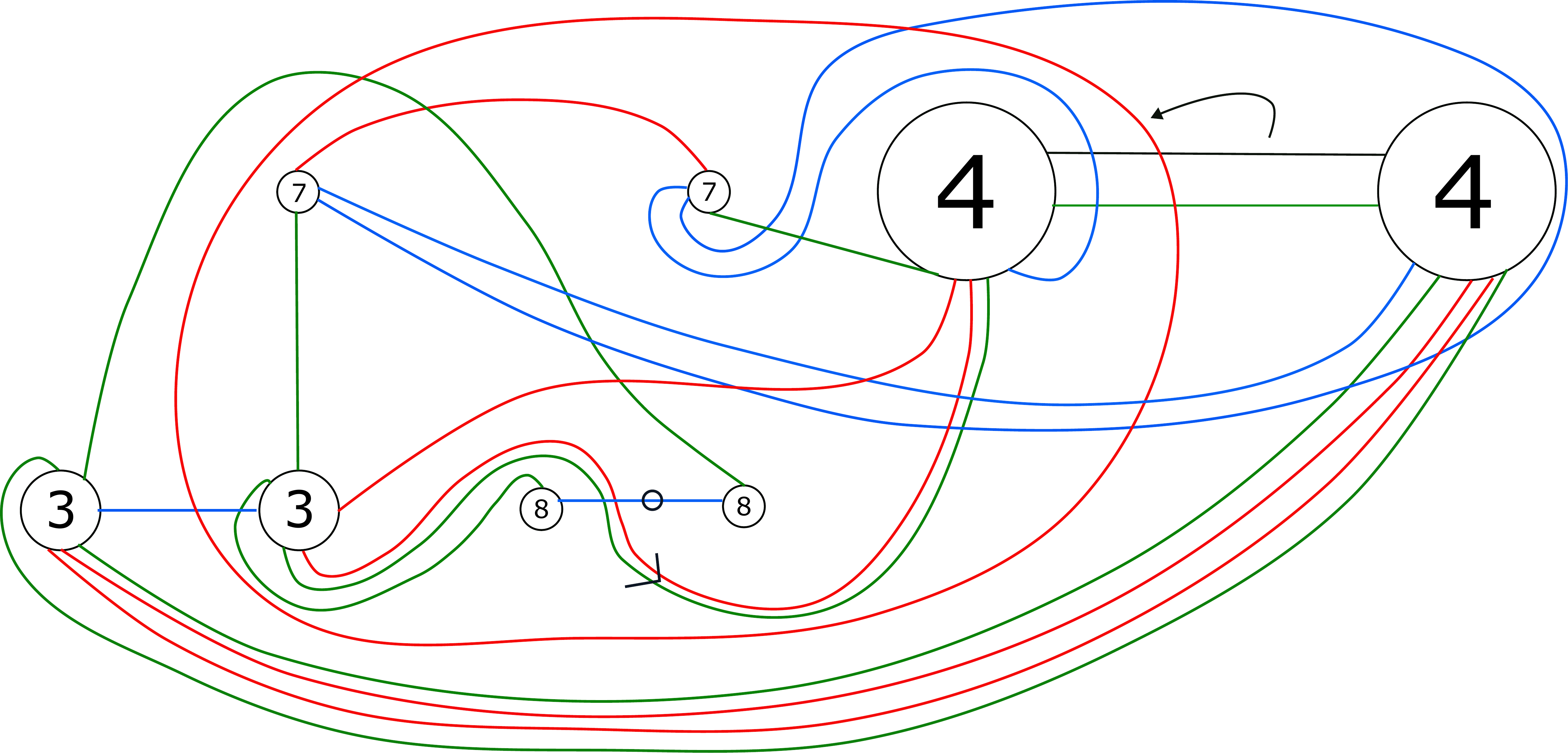}
\end{center}
\setlength{\captionmargin}{50pt}
\caption{The diagram obtained from Figure \ref{fig:before6_3} by some handle slides. (before the sixth destabilization)}
\label{fig:before6_4}
\end{figure}

\begin{figure}[h]
\begin{center}
\includegraphics[width=13cm, height=20cm, keepaspectratio, scale=1]{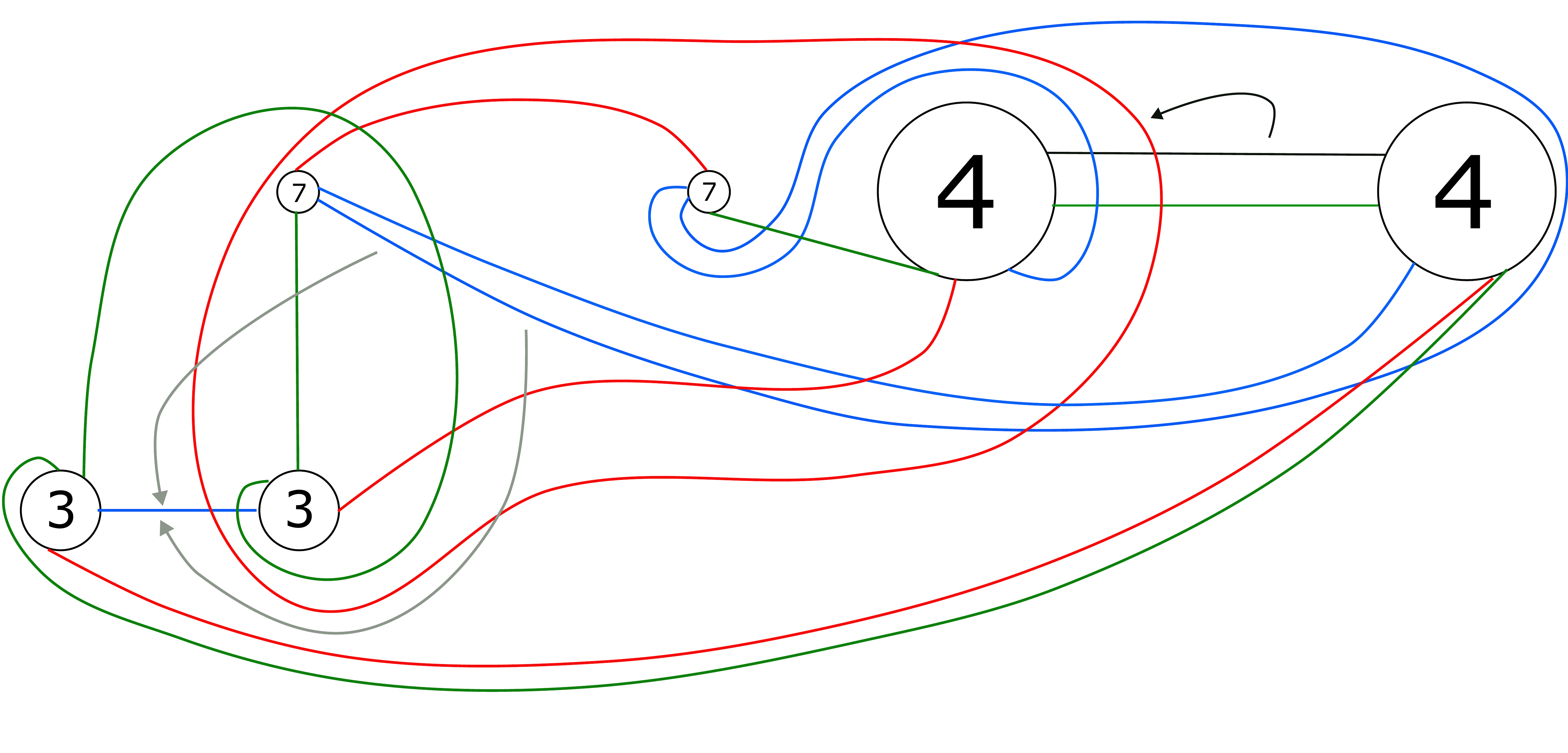}
\end{center}
\setlength{\captionmargin}{50pt}
\caption{After the sixth destabilization.}
\label{fig:after6}
\end{figure}

\subsection*{The seventh, eighth and ninth destabilizations}
In Figure \ref{fig:after6}, by performing the handle slides indicated in gray, Figure \ref{fig:before7} is obtained. By destabilizing Figure \ref{fig:before7} along the three curves labeled the arrow and circle, we have Figure \ref{fig:after7}.

In Figure \ref{fig:after7}, by performing the handle slide indicated in gray, Figure \ref{fig:before8} is obtained. By destabilizing Figure \ref{fig:before8} along the three curves labeled the arrow and circle, we have Figure \ref{fig:after8}. Note that in this step, the Dehn twist for the $\alpha$ curve is canceled for all $n$.

In Figure \ref{fig:after8}, the ninth destabilization can be clearly performed. This completes the proof.

\begin{figure}[h]
\begin{center}
\includegraphics[width=13cm, height=20cm, keepaspectratio, scale=1]{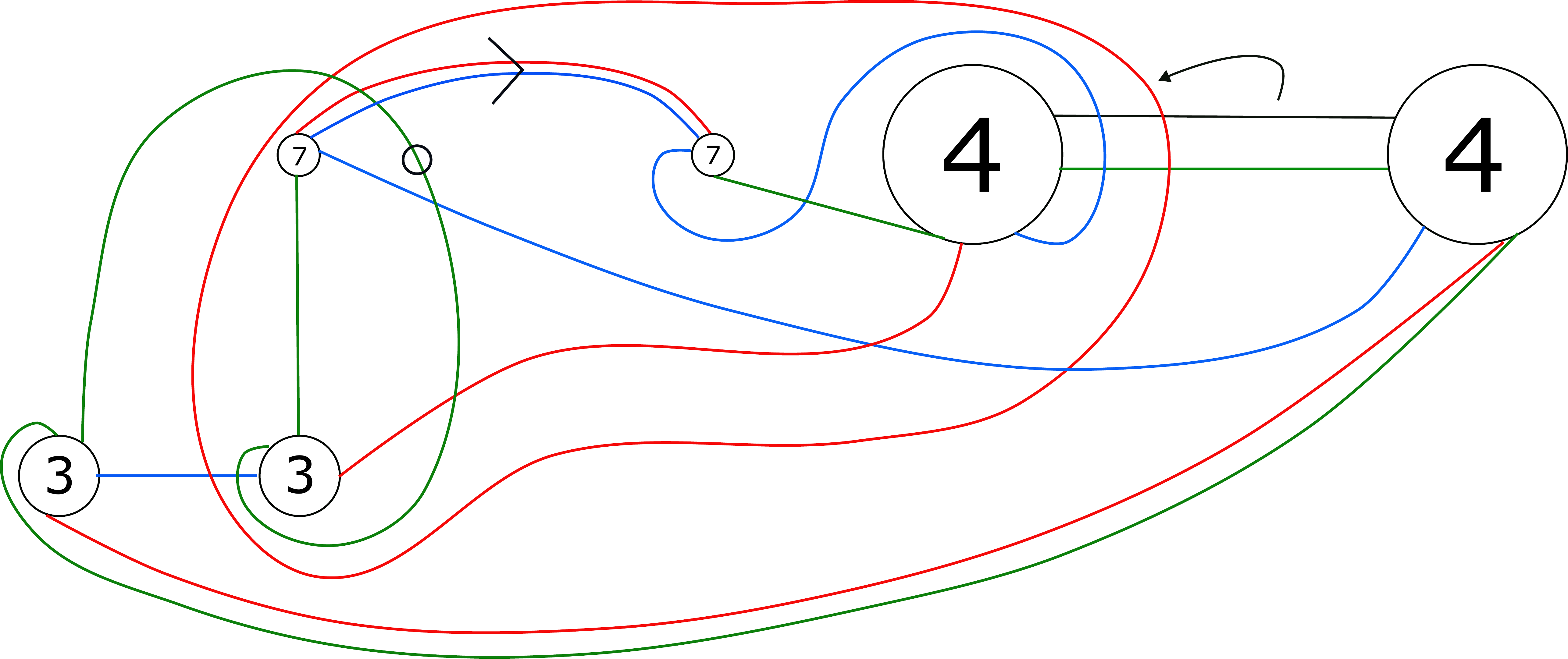}
\end{center}
\setlength{\captionmargin}{50pt}
\caption{The diagram obtained from Figure \ref{fig:after6} by some handle slides. (before the seventh destabilization)}
\label{fig:before7}
\end{figure}

\begin{figure}[h]
\begin{center}
\includegraphics[width=10cm, height=20cm, keepaspectratio, scale=1]{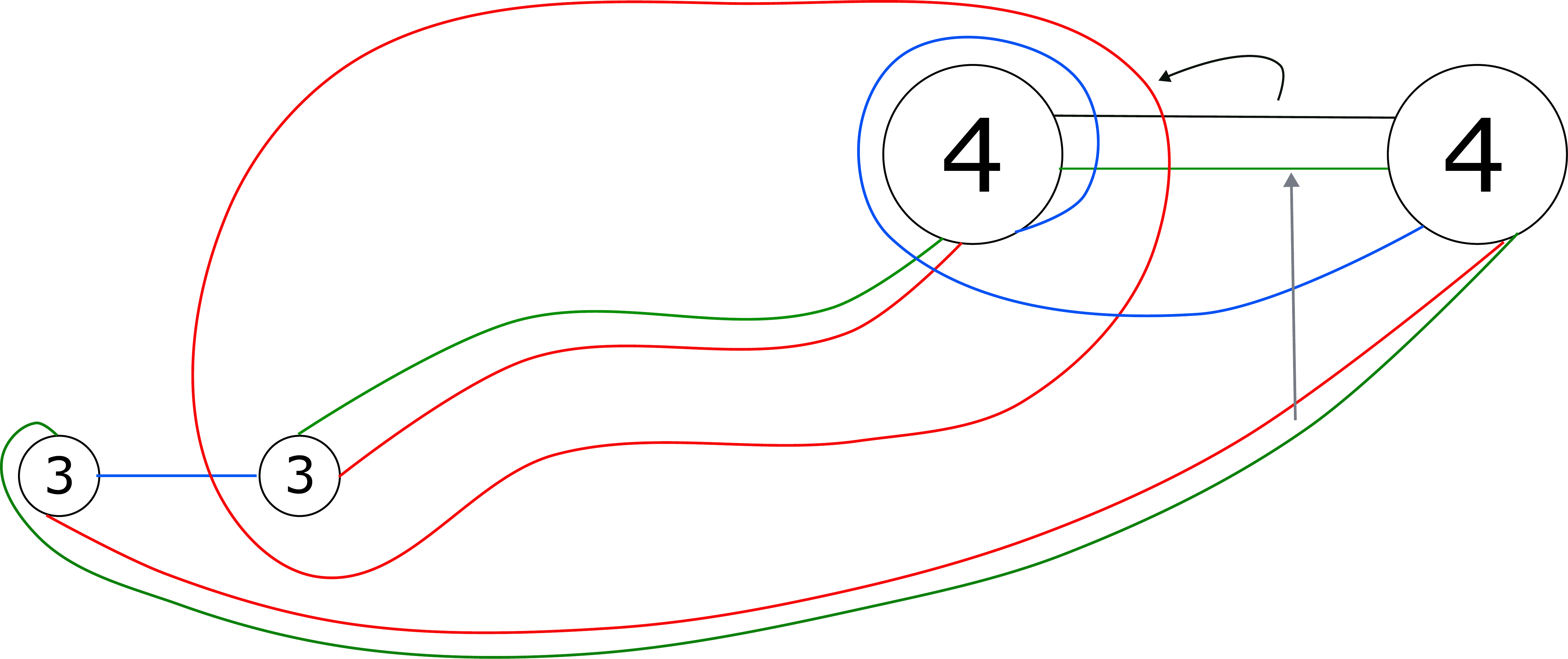}
\end{center}
\setlength{\captionmargin}{50pt}
\caption{After the seventh destabilization.}
\label{fig:after7}
\end{figure}

\begin{figure}[h]
\begin{center}
\includegraphics[width=10cm, height=20cm, keepaspectratio, scale=1]{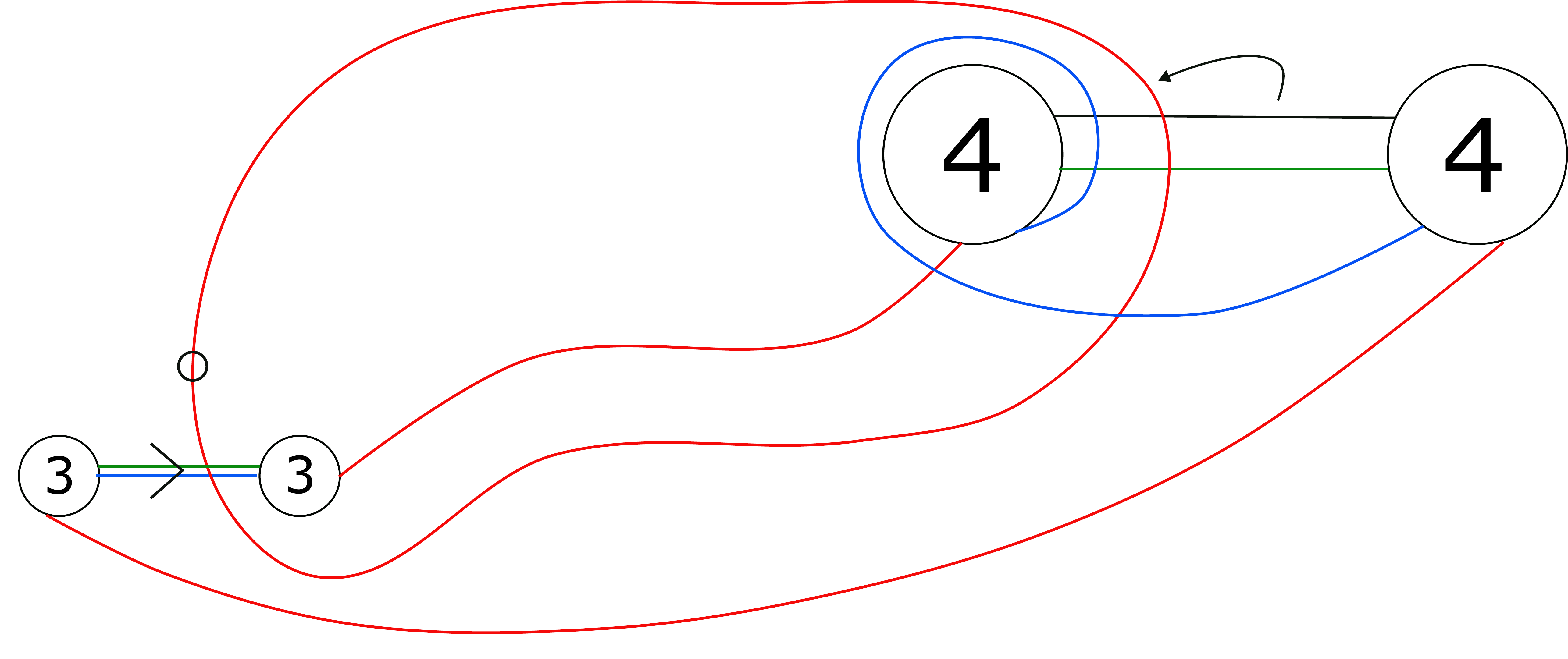}
\end{center}
\setlength{\captionmargin}{50pt}
\caption{The diagram obtained from Figure \ref{fig:after7} by some handle slide. (before the eighth destabilization)}
\label{fig:before8}
\end{figure}

\begin{figure}[h]
\begin{center}
\includegraphics[width=5cm, height=5cm, keepaspectratio, scale=1]{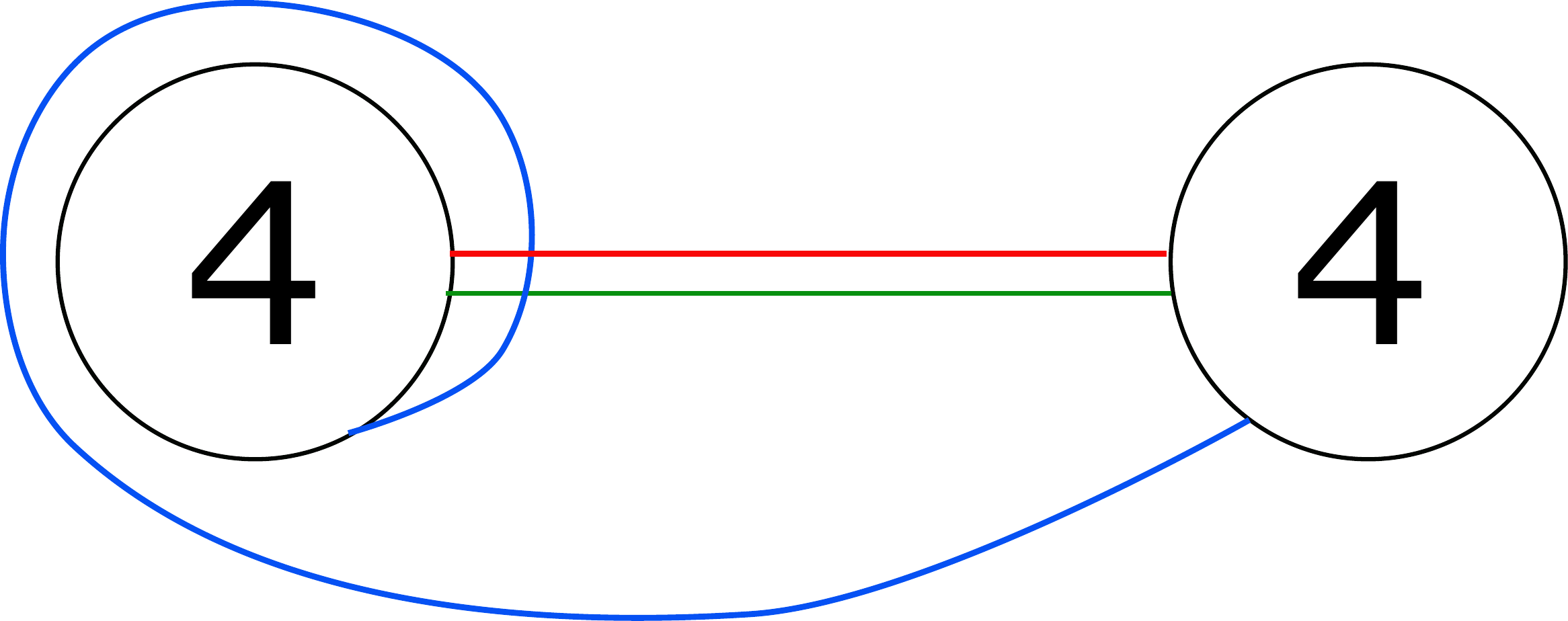}
\end{center}
\setlength{\captionmargin}{50pt}
\caption{After the eighth destabilization.}
\label{fig:after8}
\end{figure}

\end{proof}

We construct another trisection diagram of $DW^{-}(0,n+2)$ using the algorithm recalled in subsection \ref{subsec:algorithm}. The trisection diagram is shown in Figure \ref{fig:DDn2}, which is obtained from the Kirby diagram of $DW^{-}(0,n+2)$ depicted in Figure \ref{fig:DW-(l,k)} by using the algorithm. Our second theorem is as follows:

\begin{thm}\label{thm2}
The trisection diagram depicted in Figure \ref{fig:DDn2} is standard for any integer $n$.
\end{thm}

\begin{proof}
We also show the statement by destabilizing the trisection diagram continuously.

\begin{figure}[h]
\begin{center}
\includegraphics[width=4.5cm, height=4.5cm, keepaspectratio, scale=1]{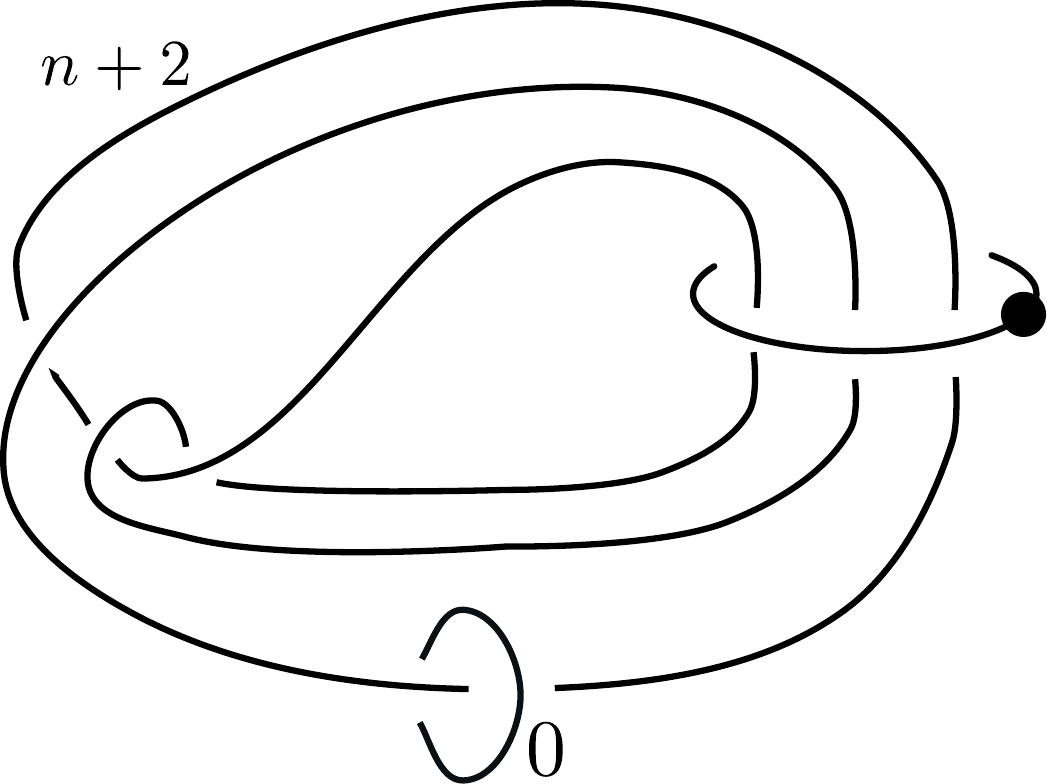}
\end{center}
\setlength{\captionmargin}{50pt}
\caption{A Kirby diagram of $DW^{-}(0,n+2)$ (omitted one 3-handle and one 4-handle). }
\label{fig:DW-(l,k)}
\end{figure}

\begin{figure}[h]
\begin{center}
\includegraphics[width=13cm, height=20cm, keepaspectratio, scale=1]{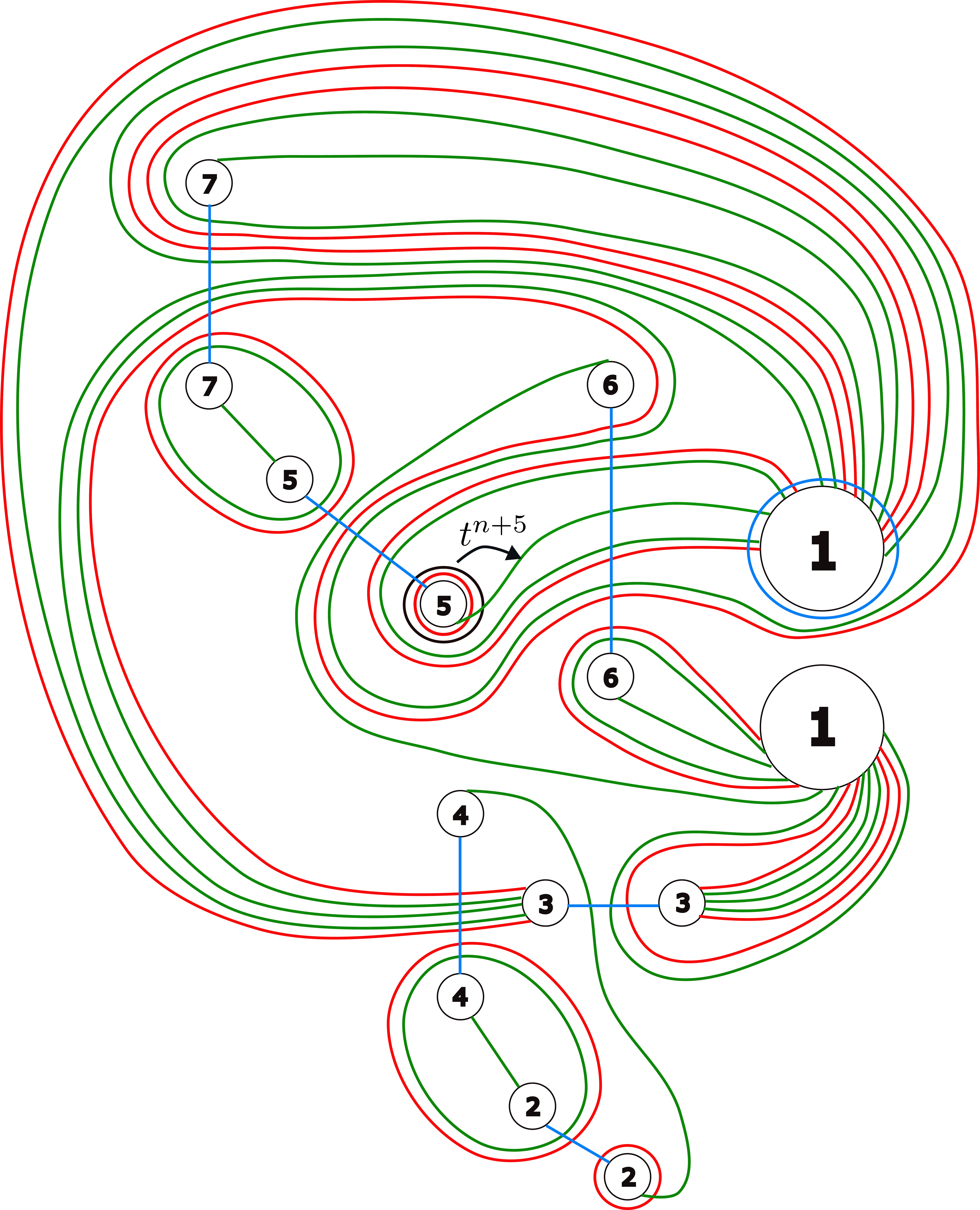}
\end{center}
\setlength{\captionmargin}{50pt}
\caption{A $(7;1,1,5)$-trisection diagram of $DW^{-}(0,n+2)$ obtained from the algorithm.}
\label{fig:DDn2}
\end{figure}

\begin{figure}[h]
\begin{center}
\includegraphics[width=13cm, height=20cm, keepaspectratio, scale=1]{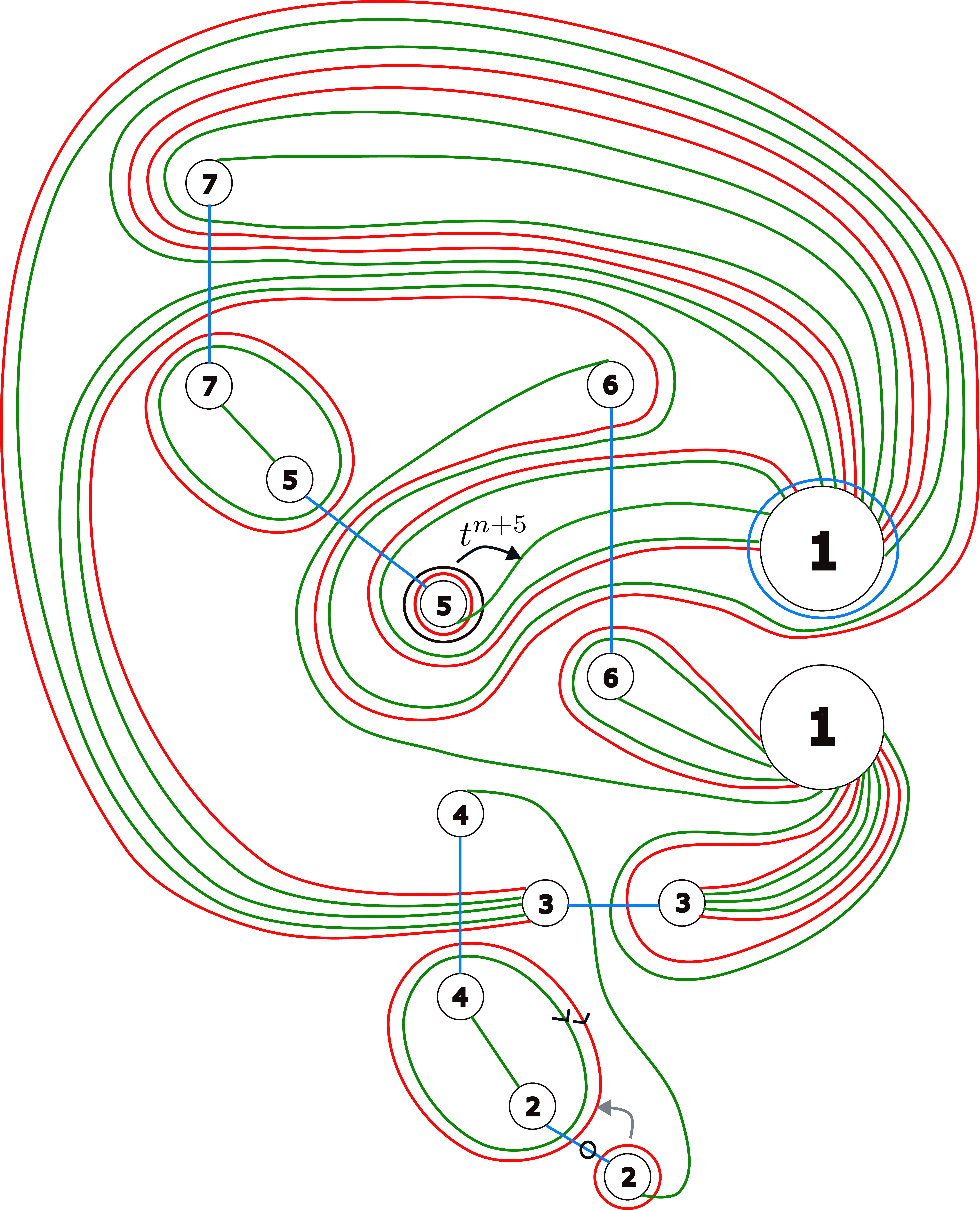}
\end{center}
\setlength{\captionmargin}{50pt}
\caption{The starting diagram in the proof.}
\label{fig:start_destabili2}
\end{figure}


\subsection*{The first destabilization}
In Figure \ref{fig:start_destabili2}, perform the handle slide indicated in gray. Then, by destabilizing the disk labeled 2, we have Figure \ref{fig:after_first2}.

\begin{figure}[h]
\begin{center}
\includegraphics[width=13cm, height=20cm, keepaspectratio, scale=1]{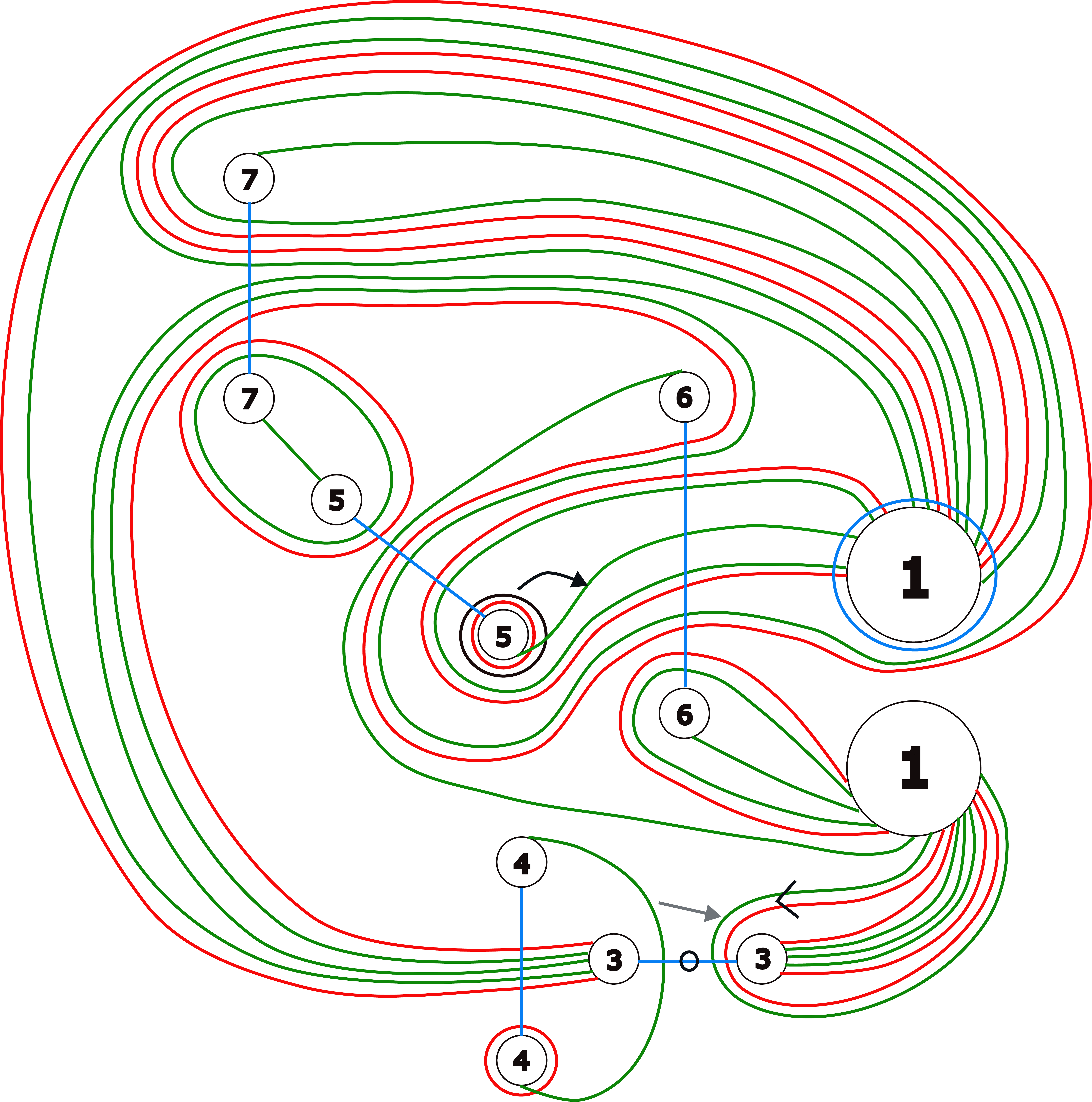}
\end{center}
\setlength{\captionmargin}{50pt}
\caption{After the first destabilization. The type is $(6;1,1,4)$.}
\label{fig:after_first2}
\end{figure}

\subsection*{The second destabilization}
In Figure \ref{fig:after_first2}, perform the handle slide indicated in gray. Then, by destabilizing the disk labeled 3, we have Figure \ref{fig:after_second2}.

\begin{figure}[h]
\begin{center}
\includegraphics[width=13cm, height=20cm, keepaspectratio, scale=1]{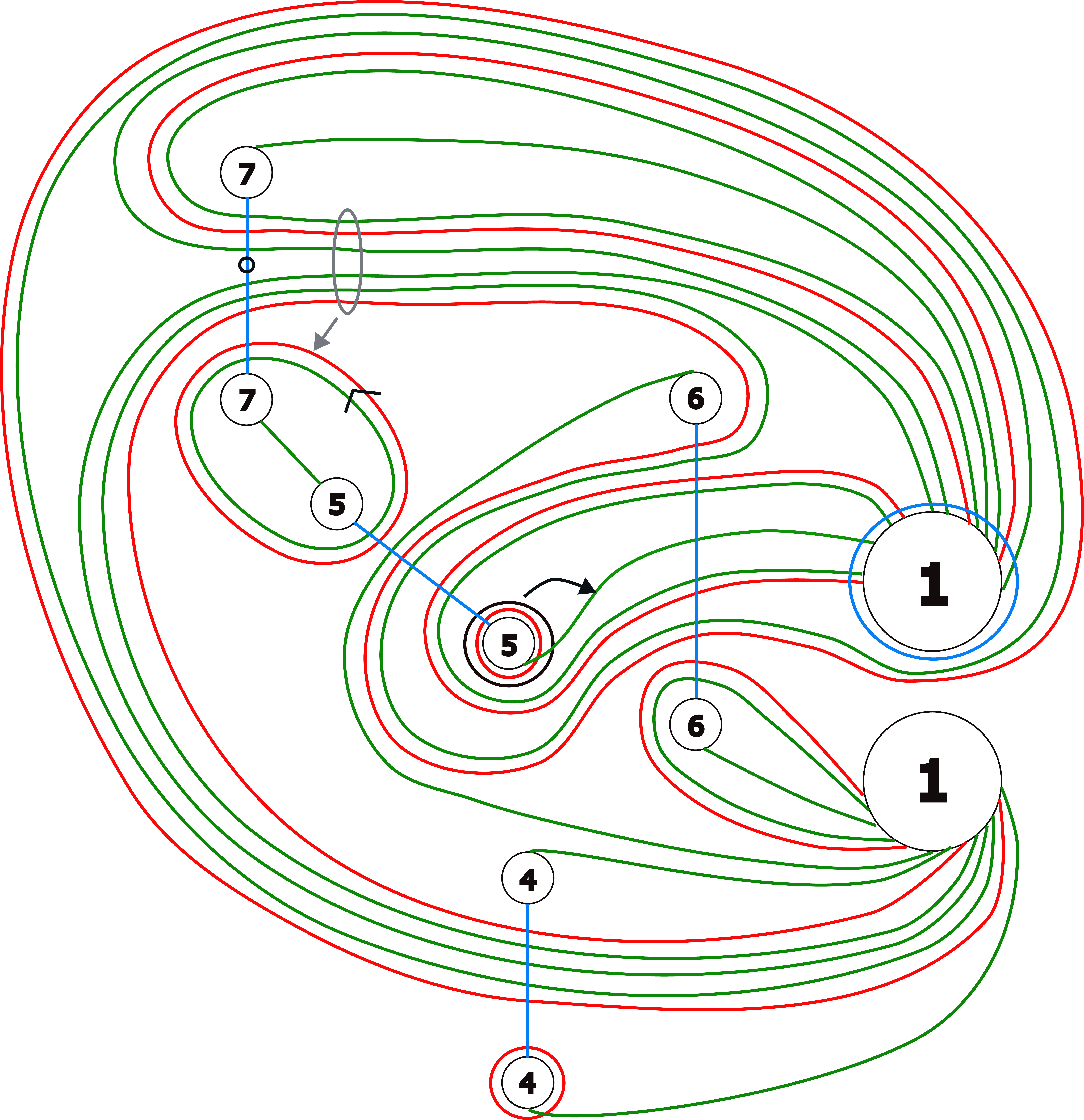}
\end{center}
\setlength{\captionmargin}{50pt}
\caption{After the second destabilization. The type is $(5;1,1,3)$.}
\label{fig:after_second2}
\end{figure}

\subsection*{The third destabilization}
In Figure \ref{fig:after_second2}, perform the handle slides indicated in gray. Then, by destabilizing the disk labeled 7, we have Figure \ref{fig:after_third2}.

\begin{figure}[h]
\begin{center}
\includegraphics[width=13cm, height=20cm, keepaspectratio, scale=1]{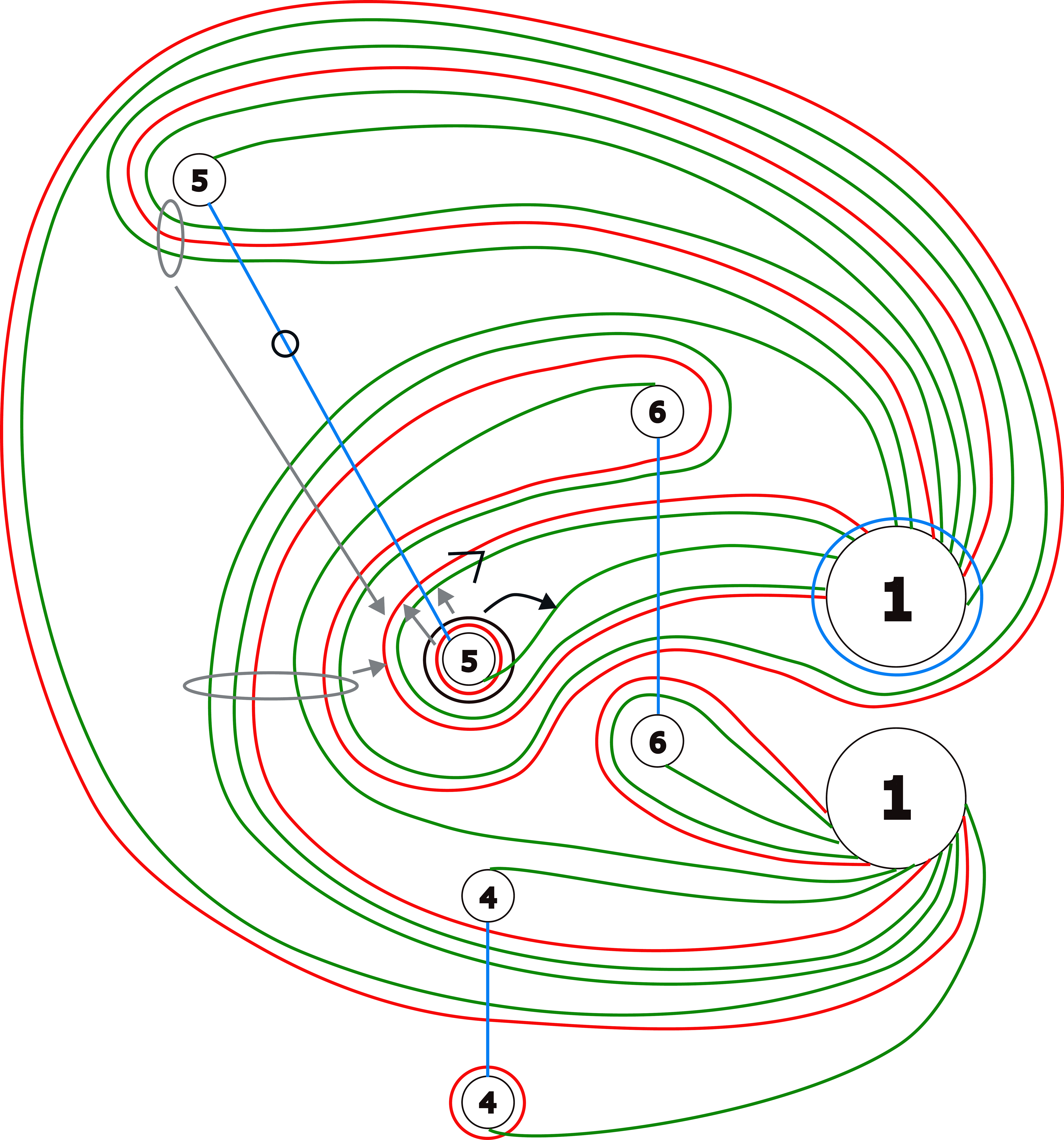}
\end{center}
\setlength{\captionmargin}{50pt}
\caption{After the third destabilization. The type is $(4;1,1,2)$.}
\label{fig:after_third2}
\end{figure}

\subsection*{The fourth destabilization}
In Figure \ref{fig:after_third2}, by performing the handle slides indicated in gray, Figure \ref{fig:before_fourth2} is obtained. Then, by destabilizing the disk labeled 5, we have Figure \ref{fig:after_fourth2}.

\begin{figure}[h]
\begin{center}
\includegraphics[width=13cm, height=20cm, keepaspectratio, scale=1]{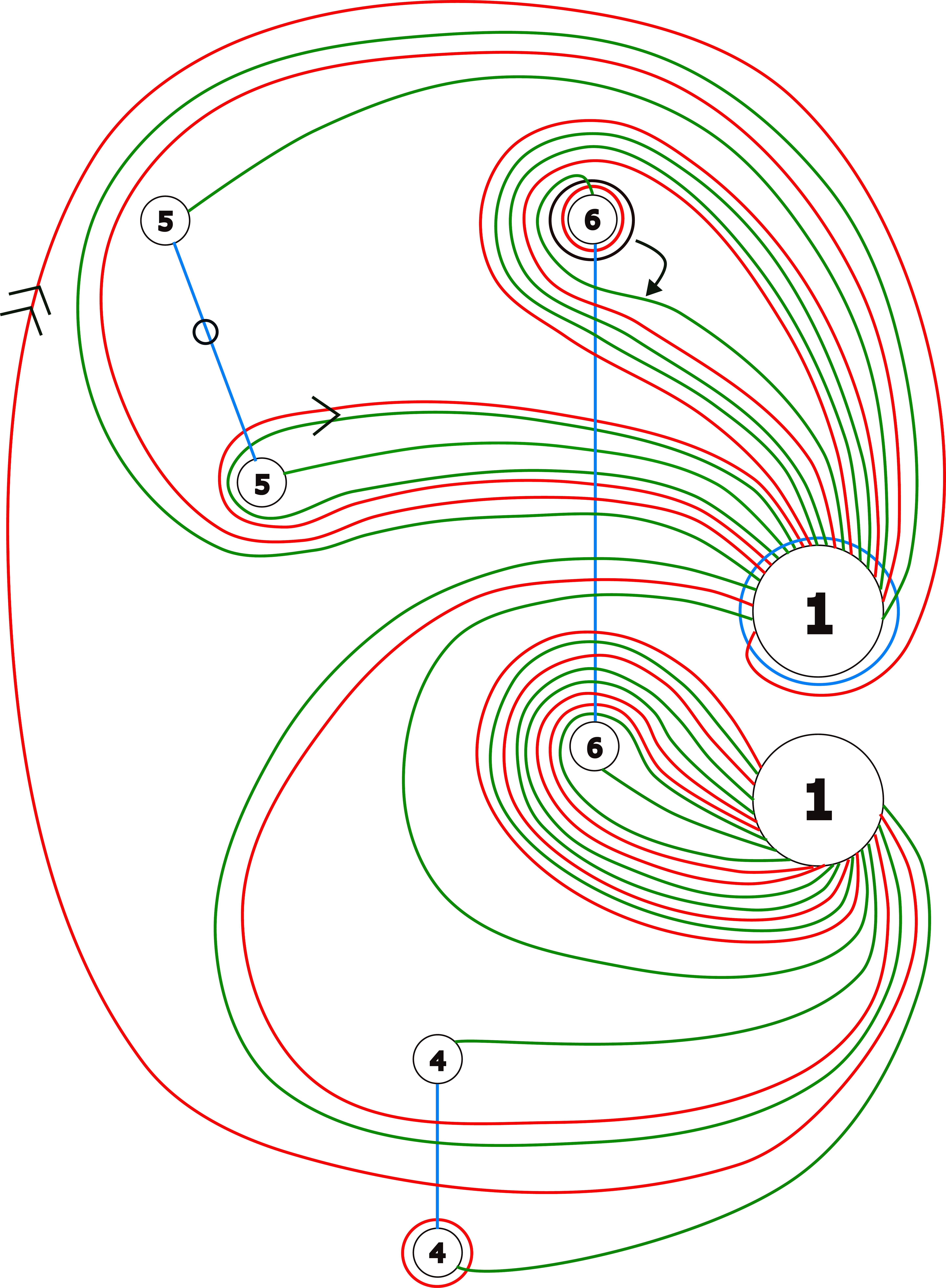}
\end{center}
\setlength{\captionmargin}{50pt}
\caption{The diagram obtained from Figure \ref{fig:after_third2} by some handle slides (before the fourth destabilization). The $\gamma$ curve parallel to the $\alpha$ curve labeled the double arrow is omitted.}
\label{fig:before_fourth2}
\end{figure}

\begin{figure}[h]
\begin{center}
\includegraphics[width=13cm, height=20cm, keepaspectratio, scale=1]{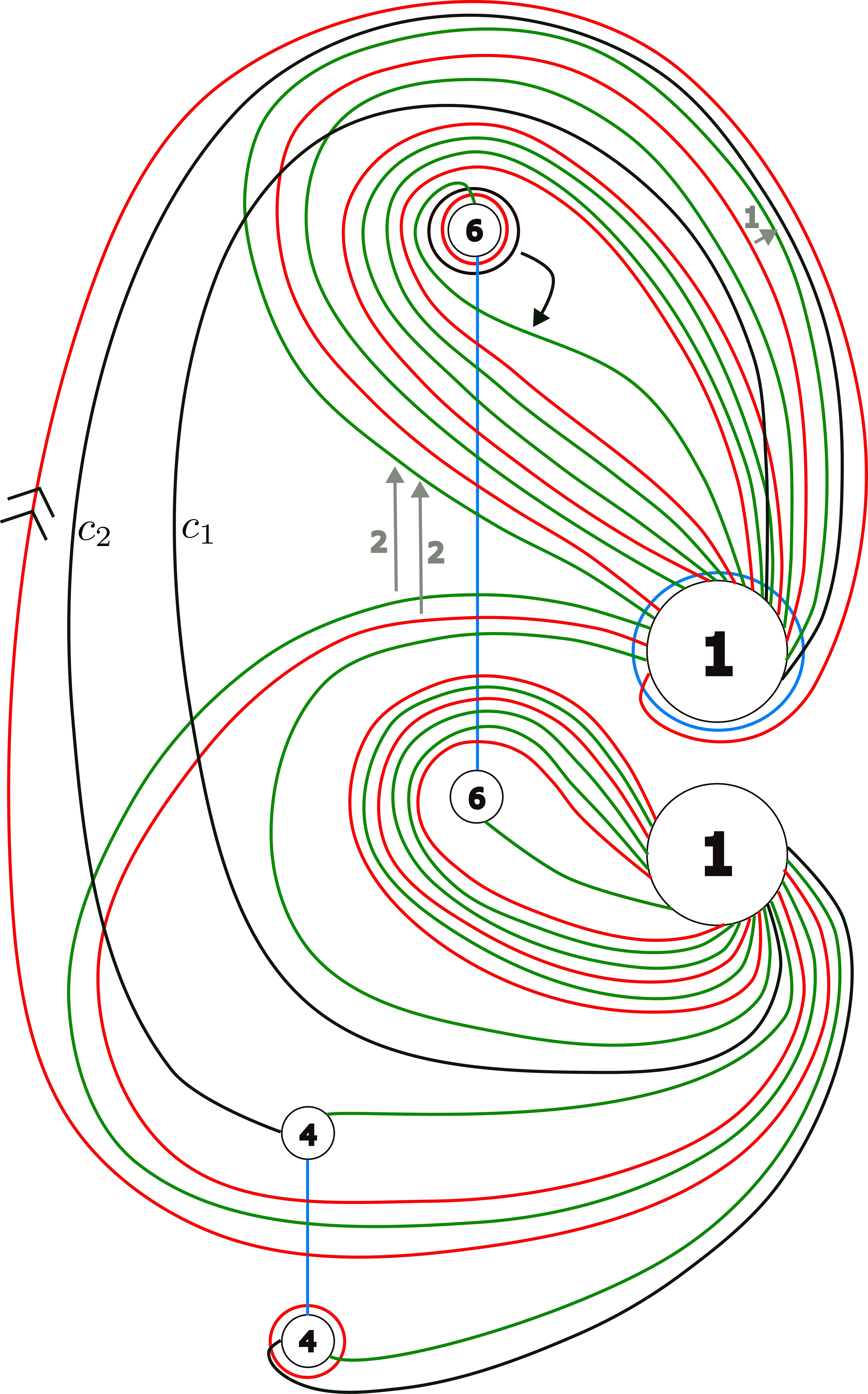}
\end{center}
\setlength{\captionmargin}{50pt}
\caption{After the fourth destabilization. The type is $(3,1)$. The $\gamma$ curve parallel to the $\alpha$ curve labeled the double arrow is omitted.}
\label{fig:after_fourth2}
\end{figure}

\subsection*{The fifth destabilization}
In Figure \ref{fig:after_fourth2}, firstly handle sliding the $\alpha$ curve labeled the double arrow over other $\alpha$ curves so that the $\alpha$ curve is parallel to the $\beta$ curve on the disk labeled 1. After that, perform the handle slides indicated in gray in order. Then, by performing the Dehn twists $t_{c_1}^{-1} t_{c_2}$ and handle sliding $\alpha$ and $\beta$ curves on the $\alpha$ and $\beta$ curves on the disk labeled 6 appropriately, Figure \ref{fig:before_fifth_2} is obtained. In Figure \ref{fig:before_fifth_2}, by performing the handle slides indicated in gray in order, Figure \ref{fig:before_fifth2_2} is obtained. We can destabilize Figure \ref{fig:before_fifth2_2} since the $\beta$ and $\gamma$ curves labeled the arrow are parallel and the $\alpha$ curve labeled the circle geometrically intersects the parallel two curves only once. If we destabilize the disk labeled 4 in Figure \ref{fig:before_fifth2_2}, we have a $(2;1,0,1)$-trisection diagram of $S^4$. It is known by \cite[Theorem 1.2]{MR3544545} that any $(2;1,0,1)$-trisection diagram of $S^4$ is standard. This completes the proof.

\begin{figure}[h]
\begin{center}
\includegraphics[width=13cm, height=20cm, keepaspectratio, scale=1]{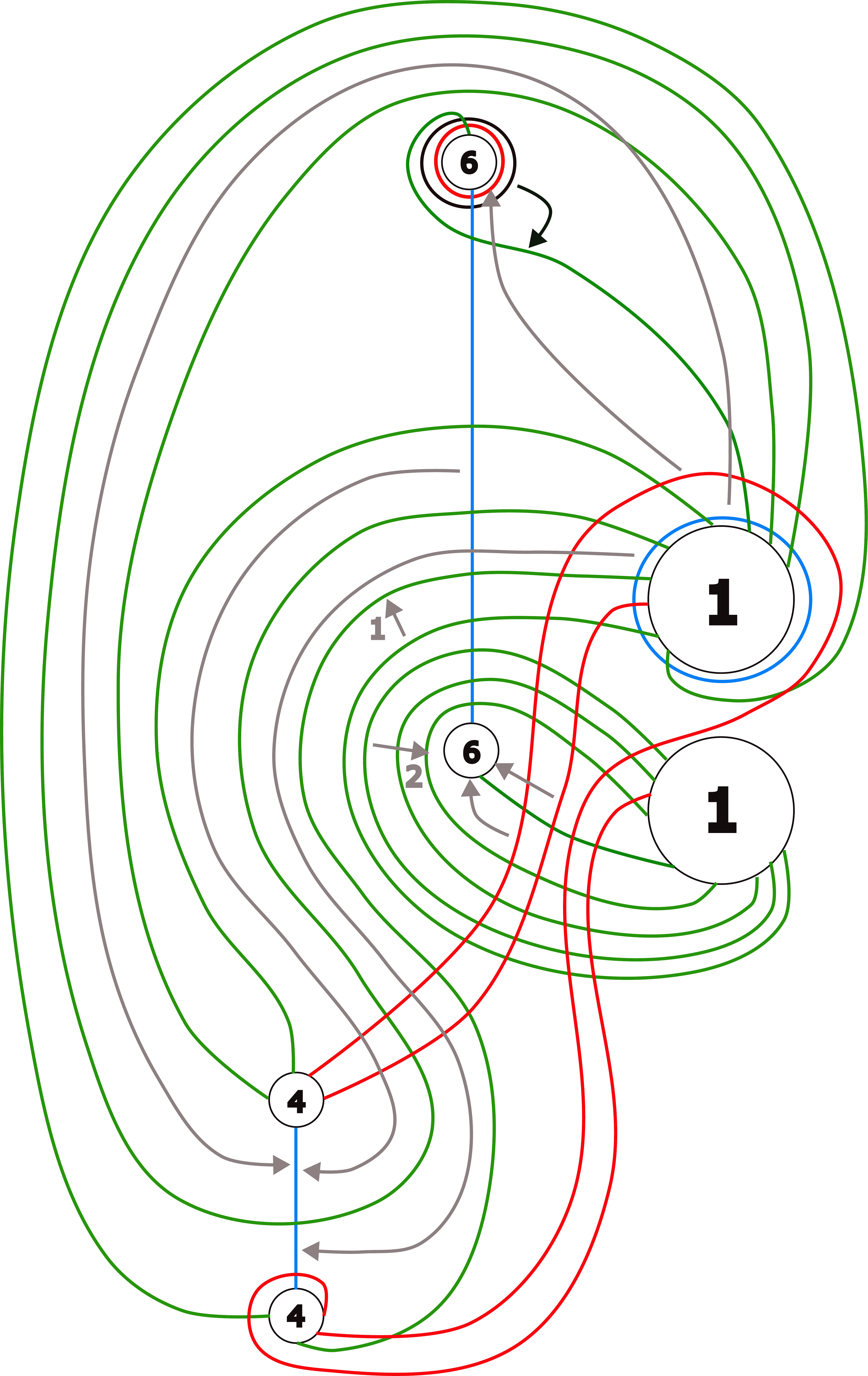}
\end{center}
\setlength{\captionmargin}{50pt}
\caption{The diagram obtained from Figure \ref{fig:after_fourth2} by some Dehn twists and handle slides.}
\label{fig:before_fifth_2}
\end{figure}

\begin{figure}[h]
\begin{center}
\includegraphics[width=13cm, height=20cm, keepaspectratio, scale=1]{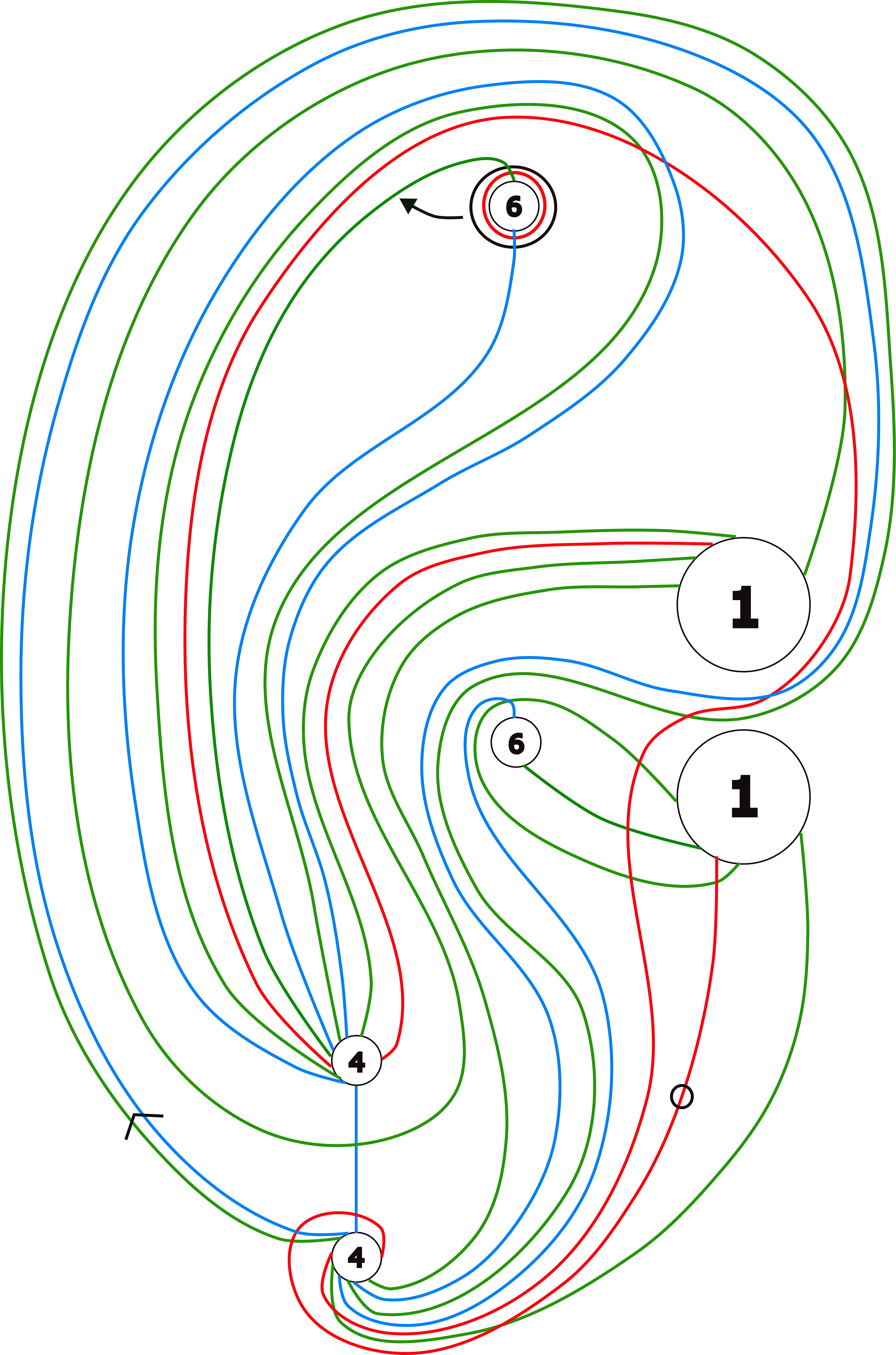}
\end{center}
\setlength{\captionmargin}{50pt}
\caption{The diagram obtained from Figure \ref{fig:before_fifth_2} by some handle slides.}
\label{fig:before_fifth2_2}
\end{figure}

\end{proof}

\begin{rem}\label{rem:main}
The $(7;1,1,5)$-trisection diagram shown in Figure \ref{fig:start_destabili2} cannot be realized as the double of any relative trisection diagram of $W^{-}(0,n+2)$.
\end{rem}

\begin{proof}
It is known \cite{MR3999550} that the type of a trisection obtained by doubling a $(g;k_1,k_2,k_3;p,b)$-relative trisection is $(2g+b-1;2k_1-\ell,2k_2-\ell,2k_3-\ell)$, where $\ell=2p+b-1$. Thus, the following conditions must be satisfied to realize a $(7;1,1,5)$-trisection of $DW^{-}(0,n+2)$: 
\[2g+b-1=7, 2k_1-\ell=1, 2k_2-\ell=1, 2k_3-\ell=5.\]
If $\ell + k \not= 2,3,4,5$, the trisection genus of $W^{-}(\ell,k)$ is 3 \cite{takahashi2023exotic}. Since $b \ge 1$, $2g+b-1=7$ does not hold if $g \ge 4$. Thus, $g=3$, and $b=2$. Then, $k_1=p+1$. Here, by $k_i \ge \ell$ for $i=1,2,3$, we have $k_1 \ge 2p+1$. So, we have $p+1 \ge 2p+1$, that is, $p \le 0$. Since $p \ge 0$, $p=0$. This relative trisection diagram induces an open book decomposition whose page is an annulus to $\partial{W^{-}(0,n+2)}$. However, $\partial{W^{-}(0,n+2)}$ is a homology 3-sphere that is not $S^3$. This completes the proof for $g=3$.

If $\ell + k = 2,3,4,5$, the trisection genus of $W^{-}(\ell,k)$ is potentially 2 \cite{takahashi2023exotic}. In the case where $g=2$, $b=4$. Then, $k_1=p+2$. Here, by $k_i \ge \ell$ for $i=1,2,3$, we have $k_1 \ge 2p+3$. So, we have $p+2 \ge 2p+3$, that is, $p \le -1$. This completes the proof for $g=2$.
\end{proof}


\bibliographystyle{amsalpha}
\bibliography{trisection, Mazur}

\end{document}